\documentclass[12pt]{amsart}

\usepackage{color}

\usepackage{amsmath,amssymb,setspace,nicefrac, yhmath, amscd,eucal}
\usepackage[active]{srcltx}
\usepackage[colorlinks, linkcolor=red, citecolor=blue, urlcolor=blue, hypertexnames=true]{hyperref}
\usepackage{amsrefs, mathrsfs}

\allowdisplaybreaks
\usepackage[all]{xy}

\newcommand{\R}{{\mathbb R}}

\newcommand{\N}{{\mathbb N}}
\newcommand{\F}{{\mathbb F}}

\newcommand{\E}{{\rm E}}
\newcommand{\ET}{{\rm ET}}

\newcommand{\cE}{{\mathcal E}}

\newcommand{\cM}{{\mathcal M}}
\newcommand{\cP}{{\mathcal P}}

\newcommand{\cF}{{\mathcal F}}
\newcommand{\minus}{\backslash}

\newcommand{\Opt}{{\rm Opt}}

\newtheorem{thm}{Theorem}[section]
\newtheorem{cor}[thm]{Corollary}
\newtheorem{lem}[thm]{Lemma}

\newtheorem{definition}[thm]{Definition}
\newtheorem{example}[thm]{Example}
\newtheorem{remark}[thm]{Remark}
\newtheorem{proposition}[thm]{Proposition}

\numberwithin{equation}{section}

\theoremstyle{definition}

\keywords{optimal transport, generalized Wasserstein space, barycenter, duality}
\subjclass[2010]{49J40, 49K21, 	49K30}
\begin{document}
\title[Weak optimal total variation transport problems]{Weak optimal total variation transport problems and generalized Wasserstein barycenters}

\author{Nhan-Phu Chung}
\address{Nhan-Phu Chung, Department of Mathematics, Sungkyunkwan University, 2066 Seobu-ro, Jangan-gu, Suwon-si, Gyeonggi-do, Korea 16419.} 
\email{phuchung@skku.edu;phuchung82@gmail.com} 
\author{ Thanh-Son Trinh}
\address{Thanh-Son Trinh, Department of Mathematics, Sungkyunkwan University, 2066 Seobu-ro, Jangan-gu, Suwon-si, Gyeonggi-do, Korea 16419.}
\email{sontrinh@skku.edu}
\date{\today}
\maketitle
\begin{abstract}
In this paper, we establish a Kantorovich duality for weak optimal total variation transport problems. As consequences, we recover a version of duality formula for partial optimal transports established by Caffarelli and McCann; and we also get another proof of Kantorovich-Rubinstein Theorem for generalized Wasserstein distance $\widetilde{W}_1^{a,b}$ proved before by Piccoli and Rossi. Then we apply our duality formula to study generalized Wasserstein barycenters. We show the existence of these barycenters for measures with compact supports. Finally, we prove the consistency of our barycenters.      
\end{abstract}
\section{Introduction}
The classical Monge-Kantorovich's optimal transport problem has been studied extensively after pioneer works of Kantorovich on 1940s \cite{Kant42,Kant48}. In connection with this problem, Wassertein distances in the space of probability measures are powerful tools to study gradient flows and partial differential equations \cite{AGS} and theory of Ricci curvature bounded below for general metric-measure spaces \cite{LV,Sturm}.

In the 2010s, various of generalizations of classical optimal transport problems and Wasserstein distances have been introduced and investigated by numerous authors \cite{ABC,BBP,CM,CPSV,GRST,KMV,Liero, PR14, PR16}. Recently, in 2021 we introduced weak optimal entropy problems \cite{CT3} which cover both optimal entropy transport problems in \cite{Liero} and weak optimal transport problems in \cite{GRST}. In \cite{CT3}, under certain conditions of entropy functionals we establish a Kantorovich duality for our weak optimal transport problem. Before stating our first main result, let us review the  weak optimal entropy problems.

Given a metric space $X$, we denote by $\cM(X)$ and $\cP(X)$ the spaces of all Borel nonnegative finite measures and probabilities measures on $X$, respectively. Let $X_1,X_2$ be Polish metric spaces and let $C:X_1\times \cP(X_2)\to [0,\infty]$ be a lower semi-continuous function satisfying that $C(x_1,\cdot)$ is convex for every $x_1\in X_1$. For every $\boldsymbol{\gamma}\in \cM(X_1\times X_2)$, we denote $(\gamma_{x_1})_{x_1\in X_1}$ its disintegration with respect to its first marginal. 
Let $F_i:[0,\infty)\to [0,\infty]$, $i=1,2$ be convex, lower semi-continuous entropy functions with their recession constants $(F_i)'_\infty:=\lim_{s\to \infty}\dfrac{F_i(s)}{s}$. Given $\mu_1\in \cM(X_1), \mu_2\in \cM(X_2)$ and $\boldsymbol{\gamma}\in \cM(X_1\times X_2)$, we define
\begin{align*}
\cF_i(\gamma_i|\mu_i):&=\int_{X_i}F_i(f_i(x_i))d\mu_i(x_i)+(F_i)'_\infty \gamma_i^\perp(X),\\
\cE(\boldsymbol{\gamma}|\mu_1,\mu_2):&=\sum_{i=1}^2\cF_i(\gamma_i|\mu_i)+\int_{X_1}C(x_1,\gamma_{x_1})d\gamma_1(x_1),
\end{align*}   
where $\gamma_1,\gamma_2$ are the first and second marginal of $\boldsymbol{\gamma}$, and $\gamma_i=f_i\mu+\gamma_i^\perp$ is the Lebesgue decomposition of $\gamma_i$ with respect to $\mu_i$.   

Our weak optimal entropy transport problem is defined as
\begin{align}
\label{P-primal}
\cE(\mu_1,\mu_2):=\inf_{\boldsymbol{\gamma}\in \cM(X_1\times X_2)} \cE(\boldsymbol{\gamma}|\mu_1,\mu_2). 
\end{align}
Similarly to optimal entropy transport problems in \cite{Liero}, to handle with problem \eqref{P-primal} we often assume that $F_i$ is superlinear, i.e. $(F_i)'_\infty=+\infty$ for $i=1,2$. This assumption makes the problems easier as we can get rid of the part $(F_i)'_\infty \gamma_i^\perp(X)$ in the expression of $\cF_i$.

In the first part of the paper, we investigate problem \eqref{P-primal} for a special case that $ F_i$ is not superlinear, $i=1,2$. Given $a,b>0$, we consider the total variation entropy function $F_i(s):=a|s-1|$, $i=1,2$ and the cost function $b\cdot C$. In this case, problem \eqref{P-primal} will become
\begin{align}
\label{P-primal1}
\E^{a,b}(\mu_1,\mu_2):=\inf_{\boldsymbol{\gamma}\in \cM(X_1\times X_2)}\E^{a,b}(\boldsymbol{\gamma}|\mu_1,\mu_2),
\end{align}
where $\E^{a,b}\left(\boldsymbol{\gamma}|\mu_1,\mu_2\right):=a\left\vert \mu_1-\gamma_1 \right\vert+a\left\vert \mu_2-\gamma_2 \right\vert+b\int_{X_1}C(x_1,\gamma_{x_1})d\gamma_1(x_1).$

As $F_i$ is not superlinear, to deal with problem \eqref{P-primal1} we need techniques being different from \cite{CT3} and \cite{Liero}. We define
\begin{align*}
\Phi_I:=\bigg\{&(\varphi_1,\varphi_2)\in C_b(X_1)\times C_b(X_2):\varphi_1(x_1),\varphi_2(x_2)\geq -a\text{ for every } x_i\in X_i,i=1,2\\
&\text{and }\varphi_1(x_1)+q(\varphi_2)\leq b\cdot C(x_1,q) \text{ for every }x_1\in X_1, q\in\cP(X_2)\bigg\}.
\end{align*}
Next, we define the functional $J:\mathbb{R}\rightarrow (-\infty,+\infty]$ by 
\begin{align}\label{F-definition of J}
J(\phi)=\sup_{s>0}\frac{\phi-a\vert 1-s\vert}{s}=\left\lbrace\begin{array}{ll}
+\infty &\text{ if }\phi>a,\\
\phi &\text{ if }-a\leq \phi\leq a,\\
-a &\text{ otherwise}.
\end{array}\right.
\end{align}
Then we define 
\begin{align*}
\Phi_J:=\bigg\{&(\varphi_1,\varphi_2)\in C_b(X_1)\times C_b(X_2):\varphi_1(x_1),\varphi_2(x_2)\leq a\text{ for every } x_i\in X_i,i=1,2\\
&\text{and }J(\varphi_1(x_1))+q(J(\varphi_2))\leq b\cdot C(x_1,q) \text{ for every }x_1\in X_1, q\in\cP(X_2)\bigg\}.
\end{align*}
Our main result for the first part is a Kantorovich duality of problem \eqref{P-primal1}.
\begin{thm}\label{T-duality for weak cost function}
Let $X_1, X_2$ be locally compact, Polish metric spaces. Let $C:X_1\times \cP(X_2)\to [0,\infty]$ be a lower semi-continuous function such that $C(x_1,\cdot)$ is convex for every $x_1\in X_1$. Then for every $\mu_i\in \cM(X_i),i=1,2$ we have 
\begin{align*}
\E^{a,b}(\mu_1,\mu_2)&=\sup_{(\varphi_1,\varphi_2)\in \Phi_I}\sum_{i=1}^2\int_{X_i}I(\varphi_i(x_i))d\mu_i(x_i)\\
&=\sup_{(\varphi_1,\varphi_2)\in \Phi_J}\sum_{i=1}^2\int_{X_i}\varphi_i(x_i)d\mu_i(x_i),
\end{align*}
where \begin{align}\label{F-definition of I}
I(\varphi):=\inf_{s\geq 0}\left(s\varphi+a\vert 1-s\vert\right)= \left\lbrace\begin{array}{cl}
a &\text{ if } \varphi>a\\
\varphi &\text{ if } -a\leq \varphi\leq a\\
-\infty &\text{ otherwise }
\end{array}\right..
\end{align}
\end{thm}
Now we present consequences of Theorem \ref{T-duality for weak cost function}. The first one is that we can get a version of \cite[Corollary 2.6]{CM}. Let $X_1=X_2=X$ be a Polish space, $\mu_1,\mu_2\in \cM(X)$, $a,b>0$ and $c_1: X\times X\to [0,+\infty]$ be a lower semi-continuous function. Attach an isolated point $\hat{\infty}$ to $X$, and extend the cost function
\begin{align}\label{F-definition of extending function}
\hat{c}_1(x,y):= \left\lbrace\begin{array}{cl}
b\cdot c_1(x,y) &\text{ if } x\neq \hat{\infty} \mbox{ and } y \neq \hat{\infty}, \\
a &\text{ if }x\in X, y=\hat{\infty}\text{ or }x=\hat{\infty},y\in X,\\
0&\text{ otherwise, }
\end{array}\right.
\end{align}
and measures $\mu_1,\mu_2$ to $\hat{X}:=X\cup \{\hat{\infty}\}$ by adding a Dirac measure at infinity: $\hat{\mu}_1:=\mu_1+\vert\mu_2\vert\delta_{\hat{\infty}}$, $\hat{\mu}_2:=\mu_2+\vert\mu_1\vert\delta_{\hat{\infty}}$. Then the measures  $\hat{\mu}_1$ and $ \hat{\mu}_2$ have the same masses. We define 
$$\Gamma(\ \hat{\mu}_1,  \hat{\mu}_2):=\bigg\{\hat{\boldsymbol{\gamma}}\in \cM(\hat{X}\times \hat{X}):\hat{\boldsymbol{\gamma}}(A\times \hat{X})=\hat{\mu}_1(A),\hat{\boldsymbol{\gamma}}(\hat{X}\times A )=\hat{\mu}_2(A) \mbox{ for Borel } A\subset \hat{X}\bigg\}.$$
Then we will get a version of \cite[Corollary 2.6]{CM} as follows.
\begin{cor}
\label{C-CM}
Given a locally compact, Polish metric space $X$, $\mu_1,\mu_2\in \cM(X)$, $a,b>0$, and a lower semi-continuous function $c_1: X\times X\to [0,+\infty]$. Then
$$\sup_{\substack{(\hat{\varphi}_1,\hat{\varphi}_2)\in L^1(\hat{\mu}_1)\times L^1(\hat{\mu}_2)\\\hat{\varphi}_1(x)+\hat{\varphi}_2(y)\leq \hat{c}_1(x,y)}}\sum_{i=1}^2\int_{\hat{X}}\hat{\varphi}_i(x)d\hat{\mu_i}(x)= \inf_{\hat{\boldsymbol{\gamma}}\in \Gamma(\hat{\mu}_1,\hat{\mu}_2)}\int_{\hat{X}\times \hat{X}}\hat{c}_1(x,y)d\hat{\boldsymbol{\gamma}}(x,y).$$
\end{cor}

Another consequence of Theorem \ref{T-duality for weak cost function} is that we establish a Kantorovich duality for generalized Wasserstein distance $\widetilde{W}_p^{a,b}$, and a version of Kantorovich-Rubinstein Theorem for generalized Wasserstein distance $\widetilde{W}_1^{a,b}$.  

Let $(X,d)$ be a metric space. For a function $f:X\to \R$, we denote 
$$\|f\|_{Lip}:=\sup_{x,y\in X,x\neq y}\frac{|f(x)-f(y)|}{d(x,y)}.$$
\begin{cor}
\label{C-flat metrics}
Let $(X,d)$ be a locally compact and Polish metric space. Then for every $a,b>0,  \mu,\nu\in \cM(X)$ and $p\geq 1$ we have 
\begin{enumerate}
\item $\widetilde{W}^{a,b}_p(\mu,\nu)^p=\sup\limits_{(\varphi_1,\varphi_2)\in\Phi_W}\bigg\{ \int_X I\left(\varphi_1(x)\right) d\mu(x)+\int_X I\left(\varphi_2(x)\right) d\nu(x)\bigg\},$
where \begin{align*}
\Phi_W:= \{(\varphi_1,\varphi_2)\in C_b(X)\times C_b(X)\;\vert\; \varphi_1(x)+\varphi_2 (y)\leq (b\cdot d(x,y))^p\text{ and }\\
\varphi_1 (x), \varphi _2(y)\geq -a,\;\forall x,y\in X\}.
\end{align*}
\item $\widetilde{W}_1^{a,b}(\mu,\nu)=\sup\bigg\{\int_X fd(\mu-\nu):f\in \F\bigg\},$
where $$\F:=\big\{f\in C_b(X), \|f\|_\infty\leq a, \|f\|_{Lip}\leq b\big\}.$$
\end{enumerate}
\end{cor} 
Note that Corollary \ref{C-flat metrics} (1) is proved for the case $p=1$ in \cite{CT}, and Corollary \ref{C-flat metrics} (2) is a main result of \cite{PR16} proved by a different method there. 	
	
In the second part of the paper, we apply Theorem \ref{T-duality for weak cost function} to study barycenters of generalized Wasserstein distances. In 2002, Sturm investigated barycenters in nonpositive curvature spaces as he showed the existence, uniqueness and contraction of barycenters in such spaces \cite{Sturm}. Because Wassertein spaces are not in the framework of nonpositive curvature spaces, to study the existence, uniqueness and properties of Wasserstein barycenters over $\R^n$, Agueh and Carlier introduced dual problems of the primal barycenter problem and used convex analysis to handle them \cite{Agueh}. It turned out that Wasserstein barycenters have applications in other fields which we only mentioned a few such as computer science, economic theory,...\cite{MR3469435,MR3862415, MR3423268}. These barycenters are also investigated further for Wasserstein spaces of Riemannian manifolds \cite{MR3590527} and locally compact geodesic spaces \cite{MR3663634}. Recently, barycenters in Hellinger-Kantorovich spaces, siblings of Wasserstein spaces, have been investigated in \cite{CP,FMS}.  

On the other hand, in 2014, Piccoli and Rossi introduced generalized Wasserstein distances \cite{PR14} and established a duality Kantorovich-Rubinstein formula and a generalized Benamou-Breiner formula for them \cite{PR16}. Combining Theorem \ref{T-duality for weak cost function} with the streamline of Agueh and Carlier's work, we study the existence and consistency of generalized Wasserstein barycenters.

More precisely, first we show the existence of generalized Wasserstein barycenters whenever starting measures have compact supports. Secondly, we introduce and investigate a dual problem of the barycenter problem. Although our barycenters are not unique, we still can establish their consistency as Boissard, Le Gouic and Loubes did in the Wasserstein case \cite{Boissard}.  
 
Our paper is organized as follows. In section 2, we review basic notations and generalized Wasserstein distances $\widetilde{W}^{a,b}_p$. In section 3, we prove Theorem \ref{T-duality for weak cost function}, Corollaries \ref{C-CM} and \ref{C-flat metrics}. In section 4, we study our primal barycenter problem and its dual problems. We also show the existence and consistency of generalized Wasserstein barycenters in this last section.

	
	\textbf{Acknowledgements:} Part of this paper was carried out when N. P. Chung visited University of Science, Vietnam National University at Hochiminh city on summer 2019. He is grateful math department there for its warm hospitality. The authors were partially supported by the National Research Foundation of Korea (NRF) grants funded by the Korea government No. NRF- 2016R1A5A1008055 , No. NRF-2016R1D1A1B03931922 and No. NRF-2019R1C1C1007107.     	  	 
\section{Preliminaries}
Let $(X,d) $ be a metric space. We denote by $\mathcal{M}(X)$ and $\mathcal{P}(X)$ the sets of all nonnegative Borel measures with finite mass and all probability Borel measures, respectively.

Given a Borel measure $\mu$, we denote its mass by $\vert \mu\vert :=\mu (X)$. In the general case, if $\mu=\mu^+-\mu^-$ is a signed Borel measure then $\vert \mu\vert :=\vert \mu^+\vert+\vert \mu^-\vert$. A set $M\subset \mathcal{M}(X)$ is bounded if $\sup_{\mu\in M}\vert \mu\vert <\infty$, and it is \textit{tight} if for every $\varepsilon>0$, there exists a compact subset $K_\varepsilon$ of $X$ such that for all $\mu\in M$, we have $\mu\left(X\backslash K_\varepsilon\right)\leq \varepsilon$.

For every $\mu_1,\mu_2\in \mathcal{M}(X)$, we say that $\mu_1$ is absolutely continuous with respect to $\mu_2$ and write $\mu_1 \ll \mu_2$ if $\mu_2(A)=0$ yields $\mu_1(A)=0$ for every Borel subset $A$ of $X$. We call that $\mu_1$ and $\mu_2$ are mutually singular and write $\mu_1 \perp \mu_2$ if there exists a Borel subset $B$ of $X$ such that $\mu_1(B)=\mu_2(X\backslash B)=0$. We write $\mu_1\leq \mu_2$ if for all Borel subset $A$ of $X$ we have $\mu_1(A)\leq \mu_2(A)$.

For every $p\geq 1$, we denote by $\mathcal{M}_p(X)$ (reps. $\mathcal{P}_p (X)$) the space of all measures $\mu\in  \mathcal{M}(X)$ (reps. $\mathcal{P}(X)$) with finite $p$-moment, i.e. there is some (and therefore any) $x_0\in X$ such that $$\int_{X}d^p\left(x,x_0\right)d\mu(x)<\infty.$$

For every measures $\mu_1,\mu_2\in \mathcal{M}(X)$, a Borel probability measure $\boldsymbol{\pi}$ on $X\times X$ is called a transference plan between $\mu_1$ and $\mu_2$ if $$\vert \mu_1\vert \boldsymbol{\pi} (A\times X)=\mu_1(A)\text{ and }\vert \mu_2\vert \boldsymbol{\pi} (X\times B)=\mu_2(B),$$
for every Borel subsets $A,B$ of $X$. We denote the set of all transference plan between $\mu_1$ and $\mu_2$ by $\Pi(\mu_1,\mu_2)$.

 Given measures $\mu_1,\mu_2\in\mathcal{M}_p(X)$ with the same mass, i.e. $\vert \mu_1\vert=\vert \mu_2\vert$. The Wasserstein distance between $\mu_1$ and $\mu_2$ is defined by $$W_p(\mu_1,\mu_2):=\left(\vert \mu_1\vert\inf_{\pi\in \Pi (\mu_1,\mu_2)}\int_{X\times X}d^p(x,y)d\boldsymbol{\pi}(x,y)\right)^{1/p}.$$
For each $\mu_1,\mu_2\in\mathcal{M}(X)$ with $|\mu_1|=|\mu_2|$, we denote by $\Opt_p(\mu_1,\mu_2)$ the set of all $\boldsymbol{\pi}\in\Pi(\mu_1,\mu_2)$ such that $W^p_p(\mu_1,\mu_2)=\vert \mu_1\vert \int_{X\times X}d^p(x,y)d\boldsymbol{\pi}(x,y)$. If $(X,d)$ is a Polish metric space, i.e. $(X,d)$ is complete and separable then $\Opt_p(\mu_1,\mu_2)$ is nonempty \cite[Theorem 1.3]{V03}.

\begin{thm} (Prokhorov's theorem)
If $(X,d)$ is a Polish metric space then a subset $M\subset \cM(X)$ is bounded and \textit{tight} if and only if $M$ is relatively compact under the weak*- topology. 
\end{thm}

We now review the definitions of the generalized Wasserstein distances. They were introduced by Piccoli and Rossi in \cite{PR14,PR16}. For convenience to establish Kantorovich duality formulas for generalized Wasserstein distances, we adapt slightly the original ones.
\begin{definition}
Let $X$ be a Polish metric space and let $a,b>0,p\geq 1$. For every $\mu_1,\mu_2\in \mathcal{M}(X)$, the generalized Wasserstein distance $\widetilde{W}^{a,b}_p$ between $\mu_1$ and $\mu_2$ is defined by $$\widetilde{W}^{a,b}_p (\mu_1,\mu_2):=\left(\inf\left\{C\left(\widetilde{\mu_1}, \widetilde{\mu_2}\right)|\, \widetilde{\mu_1},\widetilde{\mu_2}\in \mathcal{M}_p(X),\vert \widetilde{\mu_1}\vert =\vert \widetilde{\mu_2}\vert \right\}\right)^{1/p},$$ 
where $C\left(\widetilde{\mu_1}, \widetilde{\mu_2}\right)= a\left\vert \mu_1-\widetilde{\mu_1}\right\vert+a\left\vert \mu_2-\widetilde{\mu_2}\right\vert+b^p\,W_p^p\left(\widetilde{\mu_1},\widetilde{\mu_2}\right).$
\end{definition}
The following results can be adapted from the proofs of \cite[Proposition 1 and Theorem 3]{PR14}.
\begin{proposition}
\label{P-less measures for general Wassertein spaces}
(\cite[Proposition 1]{PR14})
If $X$ is a Polish metric space then $\left(\mathcal{M}(X), \widetilde{W}^{a,b}_p\right)$ is a metric space. Moreover, there exists $\widetilde{\mu_1},\widetilde{\mu_2}\in \mathcal{M}_p(X)$ such that $\vert \widetilde{\mu_1}\vert = \vert \widetilde{\mu_2}\vert, \widetilde{\mu_1}\leq \mu_1, \widetilde{\mu_2}\leq \mu_2$ and $\widetilde{W}^{a,b}_p (\mu_1,\mu_2)^p =  C\left(\widetilde{\mu_1}, \widetilde{\mu_2}\right)$.
\end{proposition}

If measures $\widetilde{\mu_1}, \widetilde{\mu_2}\in \mathcal{M}_p(X)$ have the same mass such that $\widetilde{W}^{a,b}_p (\mu_1,\mu_2)^p =  C\left(\widetilde{\mu_1}, \widetilde{\mu_2}\right)$ then we say that $\left(\widetilde{\mu_1}, \widetilde{\mu_2}\right)$ is an optimal for $\widetilde{W}^{a,b}_p (\mu_1,\mu_2)$.

Let $X_1,X_2$ be Polish metric spaces. For every $\boldsymbol{\gamma}\in \cM(X_1\times X_2)$, we denote its disintegration with respect to its first marginal by $(\gamma_{x_1})_{x_1\in X_1}$. We also denote by $\gamma_1$ and $\gamma_2$ be the first and second marginals of $\boldsymbol{\gamma}$, i.e. 
$$\gamma_1(B_1)=\boldsymbol{\gamma}(B_1\times X_2) \mbox{ and } \gamma_2(B_2)=\boldsymbol{\gamma}(X_1\times B_2) \mbox{ for Borel sets } B_i\subset X_i.$$	
\section{Weak optimal total variation transport problems}

Let $C:X_1\times \cP(X_2)\to [0,\infty]$ be a lower semi-continuous function such that for every $x_1\in X_1$ we have 
	$$C(x_1,tq_1+(1-t)q_2)\leq tC(x_1,q_1)+(1-t)C(x_1,q_2),$$
	for every $t\in [0,1], q_1, q_2\in \cP(X_2)$.

For every $a,b>0,\mu_i\in\cM (X_i),i=1,2$ and every $\boldsymbol{\gamma}\in \cM(X_1\times X_2)$, we recall 
$$\E^{a,b}\left(\boldsymbol{\gamma}|\mu_1,\mu_2\right):=a\left\vert \mu_1-\gamma_1 \right\vert+a\left\vert \mu_2-\gamma_2 \right\vert+b\int_{X_1}C(x_1,\gamma_{x_1})d\gamma_1(x_1)$$
Then for every $\mu_i\in\cM(X_i),i=1,2$ we have $$\E^{a,b}(\mu_1,\mu_2):=\inf_{\boldsymbol{\gamma}\in\cM(X_1\times X_2)}\E^{a,b}(\boldsymbol{\gamma}|\mu_1,\mu_2)=\inf_{\boldsymbol{\gamma}\in M}\E^{a,b}(\boldsymbol{\gamma}|\mu_1,\mu_2),$$
where $M:=\{\boldsymbol{\gamma}\in\cM(X_1 \times X_2)|\int_XC(x_1,\gamma_{x_1})d\gamma_1(x_1)<\infty\}.$

\begin{lem}\label{L-inf in M=inf in M^<}
Let $X_1, X_2$ be Polish metric spaces and $a,b>0$. For every $\mu_1\in\cM(X_1)$ and $\mu_2\in \cM(X_2)$ we have $$\E^{a,b}(\mu_1,\mu_2)=\inf_{\boldsymbol{\gamma}\in M}\E^{a,b}(\boldsymbol{\gamma}|\mu_1,\mu_2)=\inf_{\boldsymbol{\gamma}\in M^\leq (\mu_1,\mu_2)}\E^{a,b}(\boldsymbol{\gamma}|\mu_1,\mu_2),$$
where $M^\leq (\mu_1,\mu_2):=\{\boldsymbol{\gamma}\in M|\gamma_i\leq \mu_i,i=1,2\}.$
\end{lem}
\begin{proof}
It is clear that we only need to prove that $$\inf_{\boldsymbol{\gamma}\in M}\E^{a,b}\left(\boldsymbol{\gamma}|\mu_1,\mu_2\right)\geq \inf_{\boldsymbol{\gamma}\in M^\leq \left(\mu_1,\mu_2\right)}\E^{a,b}\left(\boldsymbol{\gamma}|\mu_1,\mu_2\right).$$
For any $\boldsymbol{\alpha}\in M$, let $\alpha_1,\alpha_2$ be the first and second marginals of $\boldsymbol{\alpha}$. Suppose that $\alpha_1=f\mu_1+\mu_1^\perp$ is the Lebesgue decomposition of $\alpha_1$ with respect to $\mu_1$. We define $\overline{\alpha}_1:= \min\{f,1\}\mu_1.$ Then $\overline{\alpha}_1\leq \mu_1$ and $\overline{\alpha}_1\leq \alpha_1$. By the Radon-Nikodym Theorem we get that there exists a measurable function $g:X_1\rightarrow [0,\infty)$ such that $\overline{\alpha}_1=g\alpha_1$ and $g\leq 1$ $\alpha_1$-a.e.

Next, for every Borel subsets $A_i$ of $X_i$, $i=1,2$, we define $$\overline{\boldsymbol{\alpha}}(A_1\times A_2):= \int_{A_1\times A_2}g(x_1)d\boldsymbol{\alpha} (x_1,x_2).$$ Then $\overline{\boldsymbol{\alpha}}(A_1\times X_2)=\int_{A_1}g(x_1)d\alpha_1(x_1)=\overline{\alpha}_1(A_1)$ for every Borel subset $A_1$ of $X_1$. For any Borel subset $A_2$ of $X_2$, we define $\overline{\alpha}_2(A_2):=\int_{X_1\times A_2}g(x_1)d\boldsymbol{\alpha} (x_1,x_2)$. Then $\overline{\alpha}_2(A_2)=\overline{\boldsymbol{\alpha}}(X_1\times A_2)$. This means that $\overline{\alpha}_1$ and $\overline{\alpha}_2$ are the first and second marginals of $\overline{\boldsymbol{\alpha}}$. Since $g\leq 1$ $\alpha_1$-a.e one has $\overline{\boldsymbol{\alpha}}\leq \boldsymbol{\alpha}$. Moreover, for every Borel function $h:X_1\times X_2\rightarrow [0,+\infty]$ we have
\begin{align*}
\int_{X_1\times X_2}h(x_1,x_2)d\overline{\boldsymbol{\alpha}}(x_1,x_2)&=\int_{X_1\times X_2}h(x_1,x_2)g(x_1)d\boldsymbol{\alpha}(x_1,x_2)\\
&=\int_{X_1}\left(\int_{X_2}h(x_1,x_2)g(x_1)d\alpha_{x_1}(x_2)\right)d\alpha_1(x_1)\\
&=\int_{X_1}\left(\int_{X_2}h(x_1,x_2)d\alpha_{x_1}(x_2)\right)d\overline{\alpha}_1(x_1).
\end{align*}
Therefore, by the uniqueness of disintegration we get that $\overline{\alpha}_{x_1}=\alpha_{x_1}$ $\overline{\alpha}_1$-a.e. This yields,
\begin{align*}
\int_{X_1}C(x_1,\overline{\alpha}_{x_1})d\overline{\alpha}_1(x_1)\leq \int_{X_1}C(x_1,\alpha_{x_1})d\alpha_1(x_1).
\end{align*} 
Notice that as $\overline{\boldsymbol{\alpha}}\leq \boldsymbol{\alpha}$, we have $\overline{\alpha}_2\leq \alpha_2$.

On the other hand, putting $D:=\{x_1\in X_1: f(x_1)\leq 1\}$ then we get that \begin{align*}
\left\vert \mu_1-\alpha_1\right\vert & = \int_{X_1}\left\vert 1-f(x_1)\right\vert d\mu_1+\mu_1^\perp (X_1)\\
& = \int_{D}(1-f(x_1))d\mu_1+\int_{X_1\minus D}(f(x_1)-1)d\mu_1+\mu_1^\perp (X_1)\\
& = \int_{D} d\mu_1-\int_{D}d\overline{\alpha}_1+\int_{X_1\minus D}f(x_1)d\mu_1-\int_{X_1\minus D}d\overline{\alpha}_1+\mu_1^\perp (X_1) \\
& = \int_{D} d\mu_1-\int_{D}d\overline{\alpha}_1+\int_{X_1\minus D}d\alpha_1-\int_{X_1\minus D}d\overline{\alpha}_1-\mu_1^\perp\left(X_1\minus D\right)+\mu_1^\perp (X_1)\\
&=\int_{D} d\mu_1-\int_{D}d\overline{\alpha}_1+\int_{X_1\minus D}d\alpha_1-\int_{X_1\minus D}d\overline{\alpha}_1+\int_D d\alpha_1-\int_Dfd\mu_1\\
& = \left\vert \mu_1-\overline{\alpha}_1\right\vert +\int_{X_1\minus D}d\alpha_1-\int_{X_1\minus D}d\overline{\alpha}_1 + \int_{D}d\alpha_1-\int_{D}d\overline{\alpha}_1\\
& = \left\vert \mu_1-\overline{\alpha}_1\right\vert+\left\vert \alpha_1-\overline{\alpha}_1\right\vert.
\end{align*} 
Observe that $\left\vert \alpha_1-\overline{\alpha}_1\right\vert=\left\vert \alpha_2-\overline{\alpha}_2\right\vert$, one gets \begin{align*}
\left\vert \mu_1-\overline{\alpha}_1\right\vert+\left\vert \mu_2-\overline{\alpha}_2\right\vert=\left\vert \mu_1-\alpha_1\right\vert-\left\vert \alpha_2-\overline{\alpha}_2\right\vert+\left\vert \mu_2-\overline{\alpha}_2\right\vert\leq \left\vert \mu_1-\alpha_1\right\vert+\left\vert \mu_2-\alpha_2\right\vert .
\end{align*}
Hence, we obtain that $$\E^{a,b}\left(\boldsymbol{\alpha}|\mu_1,\mu_2\right)\geq \E^{a,b}\left(\overline{\boldsymbol{\alpha}}|\mu_1,\mu_2\right).$$
Applying this process again for $\overline{\boldsymbol{\alpha}}$, we can find a plan $\widehat{\boldsymbol{\alpha}}\in M$ with its marginals are $\widehat{\alpha}_1$ and $\widehat{\alpha}_2$ such that $\widehat{\boldsymbol{\alpha}}\leq \overline{\boldsymbol{\alpha}}$ and $$\E^{a,b}\left(\overline{\boldsymbol{\alpha}}|\mu_1,\mu_2\right)\geq \E^{a,b}\left(\widehat{\boldsymbol{\alpha}}|\mu_1,\mu_2\right);$$
and $\widehat{\alpha}_2\leq \mu_2$, $\widehat{\alpha}_1\leq \overline{\alpha}_1\leq \mu_1$. Thus, $\widehat{\boldsymbol{\alpha}}\in M^\leq \left(\mu_1,\mu_2\right)$. Therefore, we get that $$\E^{a,b}\left(\boldsymbol{\alpha}|\mu_1,\mu_2\right)\geq \E\left(\overline{\boldsymbol{\alpha}}|\mu_1,\mu_2\right)\geq \E^{a,b}\left(\widehat{\boldsymbol{\alpha}}|\mu_1,\mu_2\right)\geq \inf_{\boldsymbol{\gamma}\in M^\leq \left(\mu_1,\mu_2\right)}\E^{a,b}\left(\boldsymbol{\gamma}|\mu_1,\mu_2\right).$$
This implies that $\inf_{\boldsymbol{\gamma}\in M}\E^{a,b}\left(\boldsymbol{\gamma}|\mu_1,\mu_2\right)\geq\inf_{\boldsymbol{\gamma}\in M^\leq \left(\mu_1,\mu_2\right)}\E^{a,b}\left(\boldsymbol{\gamma}|\mu_1,\mu_2\right)$. Hence, we get the result. 
\end{proof}

For every $a,b>0$ and $(\mu_1,\mu_2)\in \cM(X_1)\times \cM(X_2)$ we denote by $\Opt^{a,b}(\mu_1,\mu_2)$ the set of all $\boldsymbol{\gamma}\in M^\leq (\mu_1,\mu_2)$ such that $\E^{a,b}(\mu_1,\mu_2)=\E^{a,b}(\boldsymbol{\gamma}|\mu_1,\mu_2)$.
\begin{lem}\label{L-Existence of optimal plan}
Let $X_1,X_2$ be Polish metric spaces. For every $a,b>0$ and $\mu_i\in \cM(X_i),i=1,2$ the set $\Opt^{a,b}(\mu_1,\mu_2)$ is a nonempty subset of $\cM(X_1\times X_2)$.
\end{lem}
\begin{proof}
From Lemma \ref{L-inf in M=inf in M^<}, we choose a sequence of $\boldsymbol{\gamma}^n\in M^\leq (\mu_1,\mu_2)$ such that
$$\lim_{n\to \infty}\E^{a,b}(\boldsymbol{\gamma}^n|\mu_1,\mu_2)= \E^{a,b}(\mu_1,\mu_2).$$
Then $\gamma^n_i\leq \mu_i$ for $i=1,2$ and every $n\in\N$. Since $\mu_i\in \cM(X_i)$ for $i=1,2$, one has $\{\gamma_1^n\}_n$ and $\{\gamma_2^n\}_n$ are tight and bounded. By \cite[Lemma 5.2.2]{AGS} one gets that $\{\boldsymbol{\gamma}^n\}_{n\in\N}$ is also tight and bounded. Thus, by Prokhorov's Theorem, passing to a subsequence we can assume that 
$\lim_{n\to\infty}\boldsymbol{\gamma}^n= \boldsymbol{\gamma}$ under the weak*-topology for some $\boldsymbol{\gamma}\in \cM(X\times X)$.

Next, for any Borel subset $A_1$ of $X_1$ we have
\begin{eqnarray*}
\gamma_1(A_1)&=&\boldsymbol{\gamma}(A_1\times X_2)\\ 
&=&\inf\{\boldsymbol{\gamma}(V): V\subset X_1\times X_2 \mbox{ open }, A_1\times X_2\subset V\}\\
&\leq& \inf\{\boldsymbol{\gamma}(U\times X_2): U\subset X_1\mbox{ open }, A_1\subset U\}.
\end{eqnarray*}
Applying \cite[Theorem 6.1 page 40]{Par} we obtain that $\boldsymbol{\gamma}(U\times X_2)\leq \liminf_{n\to\infty}\boldsymbol{\gamma}^n(U\times X_2)\leq \mu_1(U)$ for every open subset $U$ of $X_1$. This yields, $\gamma_1\leq \mu_1.$ Similarly, we also have $\gamma_2\leq \mu_2.$ Moreover, using \cite[Theorem 6.1 page 40]{Par} again we also have that $\limsup_{n\to \infty}\vert\boldsymbol{\gamma}^n\vert\leq\vert\boldsymbol{\gamma}\vert\leq \liminf_{n\to\infty} \vert\boldsymbol{\gamma}^n\vert.$
This implies that $\lim_{n\to \infty}\vert\boldsymbol{\gamma}^n\vert=\vert\boldsymbol{\gamma}\vert$. Hence, $\lim_{n\to \infty}\left\vert \mu_i-\gamma^n_1\right\vert=\left\vert \mu_i-\gamma_i\right\vert,$ for $i=1,2$.

Applying \cite[Lemma 3.5]{CT3} we obtain that $$\liminf_{n\to \infty}\int_{X_1}C(x_1,\gamma^n_{x_1})d\gamma^n_1(x_1)\geq \int_{X_1}C(x_1,\gamma_{x_1})d\gamma_1(x_1).$$
So, we get that $$a\left\vert \mu_1-\gamma_1\right\vert+a\left\vert \mu_2-\gamma_2\right\vert+b\int_{X_1}C(x_1,\gamma_{x_1})d\gamma_1(x_1)\leq \E^{a,b}\left(\mu_1,\mu_2\right).$$ 
This implies that $\Opt^{a,b}(\mu_1,\mu_2)$ is nonempty.
\end{proof}

We recall the functionals $I,J:\mathbb{R}\rightarrow [-\infty,+\infty]$
\begin{align*}
I(\varphi):&=\inf_{s\geq 0}\left(s\varphi+a\vert 1-s\vert\right)= \left\lbrace\begin{array}{cl}
a &\text{ if } \varphi>a\\
\varphi &\text{ if } -a\leq \varphi\leq a\\
-\infty &\text{ otherwise, }
\end{array}\right.\\
J(\phi)&=\sup_{s>0}\frac{\phi-a\vert 1-s\vert}{s}=\left\lbrace\begin{array}{ll}
+\infty &\text{ if }\phi>a,\\
\phi &\text{ if }-a\leq \phi\leq a,\\
-a &\text{ otherwise},
\end{array}\right.
\end{align*}
 and
 \begin{align*}
 \Phi_I:=\bigg\{&(\varphi_1,\varphi_2)\in C_b(X_1)\times C_b(X_2):\varphi_1(x_1),\varphi_2(x_2)\geq -a\text{ for every } x_i\in X_i,i=1,2\\
&\text{and }\varphi_1(x_1)+q(\varphi_2)\leq b\cdot C(x_1,q) \text{ for every }x_1\in X_1, q\in\cP(X_2)\bigg\}, 
\end{align*}
 \begin{align*}
\Phi_J:=\bigg\{&(\varphi_1,\varphi_2)\in C_b(X_1)\times C_b(X_2):\varphi_1(x_1),\varphi_2(x_2)\leq a\text{ for every } x_i\in X_i,i=1,2\\
&\text{and }J(\varphi_1(x_1))+q(J(\varphi_2))\leq b\cdot C(x_1,q) \text{ for every }x_1\in X_1, q\in\cP(X_2)\bigg\}.
\end{align*}

We also set $$\Phi_J^0:=\{\varphi=(\varphi_1,\varphi_2)\in C_0(X_1)\times C_0(X_2): \varphi\in \Phi_J\}.$$
\begin{lem}\label{L-sup J = sup I} For every $\mu_1\in \cM(X_1)$ and $\mu_2\in \cM(X_2)$ one has
 $$\sup_{(\varphi_1,\varphi_2)\in \Phi_J}\sum_{i=1}^2\int_{X_i}\varphi_i(x_i)d\mu_i(x_i)\leq\sup_{(\psi_1,\psi_2)\in \Phi_I}\sum_{i=1}^2\int_{X_i} I(\psi_i(x_i))d\mu_i(x_i).$$
\end{lem}
\begin{proof}
For every $(\varphi_1,\varphi_2)\in \Phi_J$ and $i\in \{1,2\}$ we define $\overline{\varphi}_i:=J(\varphi_i)$. Then for every $x_i\in X_i$ we have that $ \overline{\varphi}_i(x_i)\in [-a,a]$ for $i=1,2$. Thus, from \eqref{F-definition of I} one has $I(\overline{\varphi}_i)=\overline{\varphi}_i$. Moreover, by the definition of $\Phi_J$, we also get $\overline{\varphi}_1(x_1)+q(\overline{\varphi}_2)\leq b\cdot C(x_1,q)$ for every $x_1\in X_1,q\in\cP(X_2)$. Since $J$ is continuous on $(-\infty,a]$ we get that $(\overline{\varphi}_1,\overline{\varphi}_2)\in \Phi_I$. As $\overline{\varphi}_i=J(\varphi_i)\geq \varphi_i$ for $i=1,2$ we obtain that $$\sum_{i=1}^2\int_{X_i}\varphi_i(x_i)d\mu_i(x_i)\leq\sum_{i=1}^2\int_{X_i} \overline{\varphi}_i(x_i) d\mu_i(x_i)= \sum_{i=1}^2\int_{X_i} I(\overline{\varphi}_i(x_i))d\mu_i(x_i).$$
Hence, we get the result.
\end{proof}

\begin{lem}
\label{L-easy part of duality}
Suppose that $X_1,X_2$ are Polish metric spaces. For every $a,b>0$ and $\mu_i\in\mathcal{M}(X_i), i=1,2$, we have 
\begin{align*}
\E^{a,b}(\mu_1,\mu_2)\geq\sup\limits_{(\varphi_1,\varphi_2)\in\Phi_I} \sum_{i=1}^2 \int_{X_i} I\left(\varphi_i(x_i)\right) d\mu_i(x_i).
\end{align*}
\end{lem}
\begin{proof}
Let $\mu_i\in \cM(X_i),i=1,2$. Thanks to Lemma \ref{L-Existence of optimal plan}, let $\boldsymbol{\gamma}\in \Opt^{a,b}(\mu_1,\mu_2)$. Then $\gamma_i\leq \mu_i$ for $i=1,2$. By Radon-Nikodym Theorem, there exists a measurable function $f_i:X\rightarrow [0,\infty)$ such that $\gamma_i=f_i\mu_i$ and $f_i\leq 1$ $\mu_i$-a.e. Therefore, for every $\left(\varphi_1,\varphi_2\right)\in\Phi_I$ we get that \begin{align*}
\E^{a,b}\left(\mu_1,\mu_2\right) & = \sum_{i=1}^2\int_{X_i}a\left(1-f_i(x_i)\right)d\mu_i(x_i)+b\int_{X_1} C(x_1,\gamma_{x_1})d\gamma_1(x_1)\\
& \geq \sum_{i=1}^2\int_{X_i}a\left(1-f_i(x_i)\right)d\mu_i(x_i) + \int_{X_1} \left(\varphi_1(x_1)+\gamma_{x_1}(\varphi_2)\right)d\gamma_1(x_1)\\
&= \sum_{i=1}^2\int_{X_i}a\left(1-f_i(x_i)\right)d\mu_i(x_i) + \int_{X_1}\varphi_1(x_1)d\gamma_1+\int_{X_1}\int_{X_2}\varphi_2(x_2)d\gamma_{x_1}(x_2)d\gamma_1(x_1)\\
& = \sum_{i=1}^2\int_{X_i}\left(a\left(1-f_i(x_i)\right)+f_i(x_i)\varphi_i(x_i)\right)d\mu_i(x_i).
\end{align*}
Furthermore, for all $x_i\in X_i$, since $f_i(x_i)\geq 0$, $f_i\leq 1$ $\mu_i$-a.e and \eqref{F-definition of I} we get $$\int_{X_i}I(\varphi_i(x_i))d\mu_i(x_i)\leq\int_{X_i}( f_i(x_i)\varphi_i(x_i)+a(1-f_i(x_i))d\mu_i(x_i),\text{ for }i=1,2.$$ Hence, $$\E^{a,b}\left(\mu_1,\mu_2\right)\geq  \sum_{i=1}^2 \int_{X_i} I\left(\varphi_i(x_i)\right) d\mu_i(x_i),
\text{ for every }\left(\varphi_1,\varphi_2\right)\in\Phi_I.$$ Thus, we get the result.
\end{proof}

For $i=1,2$ we denote by $\cM_s(X_i)$ the space of signed Borel measures with finite mass on $X_i$. Then for every $a,b>0$ we define the functional $\ET^{a,b}:\cM_s(X_1)\times \cM_s(X_2)\to [0,+\infty]$ by
$$\ET^{a,b}(\mu_1,\mu_2)=\left\lbrace\begin{array}{ll}
\inf_{\boldsymbol{\gamma}\in M}\E^{a,b}(\boldsymbol{\gamma}|\mu_1,\mu_2)&\text{ if }(\mu_1,\mu_2)\in \cM(X_1)\times\cM(X_2),\\
+\infty &\text{ otherwise}.
\end{array}\right.$$
\begin{lem}\label{L-homogeneous, convex, lsc of W}
Let $X_1,X_2$ be Polish metric spaces and $a,b>0$. Then 
\begin{enumerate}
\item $\ET^{a,b}$ is  positively one homogeneous, convex.
\item  If moreover $X_1$ and $X_2$ are locally compact then $\ET^{a,b}$ is lower semi-continuous under the weak*-topology.
\end{enumerate}
\end{lem}
\begin{proof}
(1) Let $\mu_i\in \cM_s(X_i),i=1,2$ and $k\geq 0$. If there exists $i\in\{1,2\}$ such that $\mu_i\not\in\cM(X_i)$ then $k\ET^{a,b}(\mu_1,\mu_2)=+\infty=\ET^{a,b}(k\mu_1,k\mu_2).$ So we only need consider $(\mu_1,\mu_2)\in \cM(X_1)\times\cM(X_2)$. Let $\boldsymbol{\gamma}\in \Opt^{a,b}(k\mu_1,k\mu_2)$ then one has \begin{align*}
\ET^{a,b}(k\mu_1,k\mu_2)&=a\left\vert k\mu_1-\gamma_1\right\vert+a\left\vert k\mu_2-\gamma_2\right\vert+b\int_{X_1}C(x_1,\gamma_{x_1})d\gamma_1(x_1)\\
& =k\left( a\left\vert \mu_1-(\gamma_1/k)\right\vert+a\left\vert \mu_2-(\gamma_2/k)\right\vert+b\int_{X_1}C(x_1,\gamma_{x_1})d(\gamma_1(x_1)/k)\right)\\
&\geq k\ET^{a,b}(\mu_1,\mu_2).
\end{align*}
Similarly, we also have $k\ET^{a,b}(\mu_1,\mu_2)\geq \ET^{a,b}(k\mu_1,k\mu_2)$ and thus $\ET^{a,b}(k\mu_1,k\mu_2)=k\ET^{a,b}(\mu_1,\mu_2)$. This shows that $\ET^{a,b}$ is  positively one homogeneous.

By the homogeneity property of $\ET^{a,b}$, to show that $\ET^{a,b}$ is convex, we only need to prove that  
$$\ET^{a,b}(\mu_1,\mu_2)+\ET^{a,b}(\nu_1,\nu_2)\geq \ET^{a,b}(\mu_1+\nu_1,\mu_2+\nu_2),$$
for every $(\mu_1,\mu_2),(\nu_1,\nu_2)\in \cM_s(X_1)\times\cM_s(X_2)$. We will consider $(\mu_1,\mu_2),(\nu_1,\nu_2)\in \cM(X_1)\times\cM(X_2)$ (the other cases are trivial). Let $\boldsymbol{\gamma}\in \Opt^{a,b}(\mu_1,\mu_2)\text{ and } \overline{\boldsymbol{\gamma}}\in \Opt^{a,b}(\nu_1,\nu_2)$. By the convexity of $C(x_1,\cdot)$ and observe that $\bigg((d\gamma_1/d(\gamma_1+\overline{\gamma}_1))\gamma_{x_1}+(d\overline{\gamma}_1/d(\gamma_1+\overline{\gamma}_1))\overline{\gamma}_{x_1}\bigg)_{x_1\in X_1}$ is the disintegration of $\boldsymbol{\gamma}+\overline{\boldsymbol{\gamma}}$ with respect to $\gamma_1+\overline{\gamma}_1$,  we have that
 $$\int_{X_1}C(x_1, (\gamma+\overline{\gamma})_{x_1})d(\gamma_1+\overline{\gamma}_1)\leq \int_{X_1}C(x_1,\gamma_{x_1})d\gamma_1+\int_{X_1}C(x_1,\overline{\gamma}_{x_1})d\overline{\gamma}_1.$$
This yields,
\begin{align*}
\ET^{a,b}(\mu_1,\mu_2)+\ET^{a,b}(\nu_1,\nu_2)&=a\sum_{i=1}^2\vert(\mu_i+\nu_i)-(\gamma_i+\overline{\gamma}_i)\vert+b\int_{X_1}C(x_1,\gamma_{x_1})d\gamma_1(x_1)\\
&+b\int_{X_1}C(x_1,\overline{\gamma}_{x_1})d\overline{\gamma}_1(x_1)\\
&\geq \ET^{a,b}(\mu_1+\nu_1,\mu_2+\nu_2).
\end{align*}

(2) For $i=1,2$, let $\{\mu_i^n\}\subset\cM(X_i)$ such that $\mu_i^n\to \mu_i\in \cM(X_i)$ as $n\to\infty$ under the weak*-topology. Then $\{\mu_i^n\}$ is relatively compact and by Prokhorov's Theorem, $\{\mu_i^n\}$ is tight and bounded. For each $n\in \N$ let $\boldsymbol{\gamma}^n\in \Opt^{a,b}(\mu_1^n,\mu_2^n)$ then $\gamma_i^n\leq \mu_i^n$ for $i=1,2$. This implies that $\{\gamma_i^n\}$ is also tight and bounded. Hence, by Prokhorov's Theorem, passing to a subsequence we can assume that $\lim_{n\to\infty}\gamma_i^n=\gamma_i$ for $\gamma_i\in\cM(X_i)$. Furthermore, for every $\mu,\nu\in\cM(X_i)$ by \cite[Theorem 6.19]{Rudin} we have that $$\vert \mu-\nu\vert =\sup\bigg\{\int_{X_i}fd(\mu-\nu)|f\in C_0(X_i),\Vert f\Vert_\infty\leq 1\bigg\}.$$
From this formula we get that $\liminf_{n\rightarrow \infty}\vert \mu_{i}^n-\gamma_{i}^n\vert\geq \vert\mu_i-\gamma_i\vert$ for $i=1,2$. By \cite[Lemma 3.5]{CT3} we get that $\int_{X_1}C(x_1,\gamma_{x_1})d\gamma_1(x_1)$ is lower semi-continuous under the weak*-topology. Therefore,
$$\liminf_{n\rightarrow\infty}\E^{a,b}(\mu_{1}^n,\mu_{2}^n)\geq a\vert \mu_1-\gamma_1\vert+a\vert \mu_2-\gamma_2\vert+b\int_{X_1}C(x_1,\gamma_{x_1})d\gamma_1(x_1)\geq \E^{a,b}(\mu_1,\mu_2).$$
This means that $\E^{a,b}$ is lower semi-continuous. Therefore, $\ET^{a,b}$ is also lower semi-continuous since $\cM(X_1)\times\cM(X_2)$ is closed.
\end{proof}

We now prove the main result in this section.
\begin{proof}[Proof of Theorem \ref{T-duality for weak cost function}]
Denote by $(\ET^{a,b})^*$ the Fenchel conjugate of $\ET^{a,b}$, i.e. 
$$(\ET^{a,b})^*(\varphi_1,\varphi_2):=\sup_{(m_1,m_2)\in \cM_s(X_1)\times\cM_s(X_2)}\bigg\{\sum_{i=1}^2\int_{X_i}\varphi_i(x_i)dm_i(x_i)-\E^{a,b}(m_1,m_2)\bigg\},$$
for every $(\varphi_1,\varphi_2)\in C_0(X_1)\times C_0(X_2)$. Notice that the dual space of $C_0(X_i)$ is $\cM_s(X_i)$. By Lemma \ref{L-homogeneous, convex, lsc of W} one has $\ET^{a,b}$ is positively one homogeneous. This implies that $$(\ET^{a,b})^*(\varphi_1,\varphi_2)=\left\lbrace\begin{array}{ll}
0 &\text{ if }(\varphi_1,\varphi_2)\in \Phi_E,\\
+\infty &\text{ otherwise},
\end{array}\right.$$
where \begin{align*}
\Phi_E:=\bigg\{(\varphi_1,\varphi_2)\in C_0(X_1)\times C_0(X_2)&: \sum_{i=1}^2\int_{X_i}\varphi_i(x_i)dm_i(x_i)\leq \ET^{a,b}(m_1,m_2)\\
& \mbox{ for every } (m_1,m_2)\in \cM_s(X_1)\times\cM_s(X_2)\bigg\}.
\end{align*}
We now check that $\Phi_E=\Phi_J^0$. Let any $(\varphi_1,\varphi_2)\in \Phi_J^0$. Let $m_i\in \cM_s(X_i)$, $i=1,2$. If $\ET^{a,b}(m_1,m_2)=+\infty$ then it is clear that $\sum_{i=1}^2\int_{X_i}\varphi_i(x_i)dm_i(x_i)\leq \ET^{a,b}(m_1,m_2)$. Thus, we only consider $(m_1,m_2)\in \cM(X_1)\times \cM(X_2)$. By Lemmas \ref{L-easy part of duality} and \ref{L-sup J = sup I} we get that   $$\sum_{i=1}^2\int_{X_i}\varphi_idm_i\leq \sup_{(\phi_1,\phi_2)\in \Phi_I}\sum_{i=1}^2\int_{X_i}I(\phi_i)dm_i\leq \E^{a,b}(m_1,m_2)=\ET^{a,b}(m_1,m_2).$$
Therefore, $(\varphi_1,\varphi_2)\in \Phi_E$ and thus $\Phi^0_J\subset \Phi_E$.

Now, let any $(\varphi_1,\varphi_2)\in \Phi_E$. We will show that $(\varphi_1,\varphi_2)\in \Phi_J^0$. Denote by $\eta$ the null measure on $X_1\times X_2$. As $(\varphi_1,\varphi_2)\in \Phi_E$, for every $(m_1,m_2)\in \cM(X_1)\times\cM(X_2)$ one has $$\sum_{i=1}^2\int_{X_i} \varphi_i(x_i) dm_i(x_i)\leq \E^{a,b}(m_1,m_2)\leq \E^{a,b}(\eta|m_1,m_2)=a(\vert m_1\vert+\vert m_2\vert).$$
For every $z\in X_1$, setting $m_1:=\delta_{z}$ and $m_2$ is the null measure on $X_2$, we obtain that $\varphi_1(z)\leq a$. Similarly, we also have $\varphi_2\leq a$ on $X_2$.

On the other hand, for any $w\in X_1$ and $q\in \cP(X_2)$ putting $m_1:=\delta_{w}, m_2:=q|_B$ and $\overline{\boldsymbol{\gamma}}:=\delta_{w}\otimes q$, where $B:=\{x_2\in X_2|\varphi_2(x_2)\geq -a\}$. Then $$\varphi_1(w)+\int_B\varphi_2dq=\sum_{i=1}^2\int_{X_i}\varphi_i(x_i)dm_i(x_i)\leq \E^{a,b}(\overline{\boldsymbol{\gamma}}|m_1,m_2)=a.q(X_2\minus B)+ b\cdot C(w,q).$$
From \eqref{F-definition of J}, if $\varphi_1(w)<-a$ then $$J(\varphi_1(w))+q(J(\varphi_2))\leq -a+a=0\leq b\cdot C(w,q),$$
and if $\varphi_1(w)\geq -a$ then \begin{align*}
J(\varphi_1(w))+q(J(\varphi_2))&= \varphi_1(w)+\int_BJ(\varphi_2)dq+\int_{X_2\minus B}J(\varphi_2)dq\\
&= \varphi_1(w)+\int_B\varphi_2dq-a.q(X_2\minus B)\\
&\leq b\cdot C(w,q).
\end{align*}
Therefore, $(\varphi_1,\varphi_2)\in \Phi^0_J$ and hence $\Phi_E\subset\Phi^0_J$. Thus, $\Phi_E=\Phi^0_J$.

Moreover, by Lemma \ref{L-homogeneous, convex, lsc of W} one has $\ET^{a,b}$ is convex and lower semi-continuous. Hence, applying \cite[Proposition 3.1, page 14 and Proposition 4.1, page 18]{Ekeland} we get that $(\ET^{a,b})^{**}=\ET^{a,b}$. Therefore,
\begin{align*}
\ET^{a,b}(\mu_1,\mu_2)&=\sup_{(\varphi_1,\varphi_2)\in C_0(X_1)\times C_0(X_2)}\left\{\sum_{i=1}^2\int_{X_i}\varphi_i(x_i)d\mu_i(x_i)-(\ET^{a,b})^*(\varphi_1,\varphi_2)\right\}\\
&=\sup_{(\varphi_1,\varphi_2)\in \Phi^0_J}\sum_{i=1}^2\int_{X_i}\varphi_i(x_i)d\mu_i(x_i)\\
&\leq \sup_{(\varphi_1,\varphi_2)\in \Phi_J}\sum_{i=1}^2\int_{X_i}\varphi_i(x_i)d\mu_i(x_i).
\end{align*}
Now, using Lemmas \ref{L-sup J = sup I} and \ref{L-easy part of duality} we get the result.
\end{proof}

\begin{proof}[Proof of Corollary \ref{C-CM}] We define the cost function $C:X\times\cP(X)\to [0,\infty]$ by 
$$C(x,q):=\int_{X}c_1(x,y)dq(y),$$
for every $x\in X$ and $q\in \cP(X)$. We will check that $C$ is lower semi-continuous on $X\times \cP(X)$. Let $(x^n,q^n)\subset X\times \cP(X)$ such that $(x^n,q^n)\to (x^0,q^0)$ as $n\to\infty$. Then as $c_1$ is lower semi-continuous on $X\times X$ and non-negative, by \cite[Lemma 4.2]{CT3} we get that
$$\liminf_{n\to\infty} C(x^n,q^n)=\liminf_{n\to\infty}\int_Xc_1(x^n,y)dq^n(y)\geq \int_Xc_1(x^0,y)dq^0(y)=C(x^0,q^0).$$
This means that $C$ is lower semi-continuous on $X\times\cP(X)$.
 
Next, a one-to-one correspondence between $\boldsymbol{\gamma}\in \cM^\leq (\mu_1,\mu_2)$ and $\hat{\boldsymbol{\gamma}}\in \Gamma(\hat{\mu}_1,\hat{\mu}_2)$ is given by $$\hat{\boldsymbol{\gamma}}=\boldsymbol{\gamma}+\vert 1-f_1\vert\mu_1\otimes \delta_{\hat{\infty}}+\delta_{\hat{\infty}}\otimes\vert 1-f_2\vert\mu_2+\vert\boldsymbol{\gamma}\vert\delta_{(\hat{\infty},\hat{\infty})},$$
where $f_i$ is the Radon-Nikodym derivative of $\gamma_i$ with respect to $\mu_i$. From this and Theorem \ref{T-duality for weak cost function} we obtain that \begin{align*}
\inf_{\hat{\boldsymbol{\gamma}}\in \Gamma(\hat{\mu}_1,\hat{\mu}_2)}\int_{\hat{X}}\hat{c}_1(x,y)d\hat{\boldsymbol{\gamma}}(x,y)=\E^{a,b}(\mu_1,\mu_2)=\sup_{(\varphi_1,\varphi_2)\in \Phi_J}\sum_{i=1}^2\int_{X}\varphi_i(x)d\mu_i(x).
\end{align*}

Now, for any $(\varphi_1,\varphi_2)\in \Phi_J$ we define $\hat{\varphi_i}(x)=J(\varphi_i(x))$ if $x\in X$ and $\hat{\varphi}_i(x)=0$ if $x=\hat{\infty}$ for $i=1,2$. Then $\hat{\varphi}_i\in L^1(\hat{\mu}_i)$ for $i=1,2$. 
As $(\varphi_1,\varphi_2)\in \Phi_J$, for every $x,y\in X$ we have
\begin{align*}
J(\varphi_1(x))+J(\varphi_2(y))=J(\varphi_1(x))+\delta_y(J(\varphi_2))\leq b\cdot C(x,\delta_y)=b\cdot c_1(x,y).
\end{align*} 
Hence $\hat{\varphi}_1(x)+\hat{\varphi}_2(y)\leq \hat{c}_1(x,y)\text{ for every }x,y\in\hat{X}.$
Moreover we also have \begin{align*}
\int_{\hat{X}}\hat{\varphi}_1(x)d\hat{\mu}_1(x)&=\int_X\hat{\varphi}_1(x)d\mu_1(x)+\hat{\varphi}_1(\hat{\infty})|\mu_2|\\
&=\int_XJ(\varphi_1(x))d\mu_1(x)\\
&\geq\int_X\varphi_1(x)d\mu_1(x).
\end{align*}
Similarly, $\int_{\hat{X}}\hat{\varphi}_2(x)d\hat{\mu}_2(x)\geq\int_X\varphi_2(x)d\mu_2(x)$. Therefore, $$\sup_{(\varphi_1,\varphi_2)\in \Phi_J}\sum_{i=1}^2\int_{X}\varphi_i(x)d\mu_i(x)\leq \sup_{\substack{(\hat{\varphi}_1,\hat{\varphi}_2)\in L^1(\hat{\mu}_1)\times L^1(\hat{\mu}_2)\\\hat{\varphi}_1(x)+\hat{\varphi}_2(y)\leq \hat{c}_1(x,y)}}\sum_{i=1}^2\int_{\hat{X}}\hat{\varphi}_i(x)d\hat{\mu_i}(x).$$
This implies that \begin{align*}
\inf_{\hat{\boldsymbol{\gamma}}\in \Gamma(\hat{\mu}_1,\hat{\mu}_2)}\int_{\hat{X}}\hat{c}_1(x,y)d\hat{\boldsymbol{\gamma}}(x,y)&\leq \sup_{\substack{(\hat{\varphi}_1,\hat{\varphi}_2)\in L^1(\hat{\mu}_1)\times L^1(\hat{\mu}_2)\\\hat{\varphi}_1(x)+\hat{\varphi}_2(y)\leq \hat{c}_1(x,y)}}\sum_{i=1}^2\int_{\hat{X}}\hat{\varphi}_i(x)d\hat{\mu_i}(x)\\
&\leq \inf_{\hat{\boldsymbol{\gamma}}\in \Gamma(\hat{\mu}_1,\hat{\mu}_2)}\int_{\hat{X}}\hat{c}_1(x,y)d\hat{\boldsymbol{\gamma}}(x,y).
\end{align*}
Hence, we get the result.
\end{proof}

\begin{proof}[Proof of Corollary \ref{C-flat metrics}]
(1) Applying Theorem \ref{T-duality for weak cost function} for $X_1=X_2=X$ and $C(x,q)=\int_Xc(x,y)dq(y)$, where $c(x,y)=(b\cdot d(x,y))^p$ for every $x,y\in X$ then we get the result.

(2) We use the techniques of the proof of \cite[Theorem 1.14]{V03} to prove (2).
For every $(\psi,\varphi)\in \Phi_W$, we define $\varphi^d(x):=\inf_{y\in X}[b\cdot d(x,y)-\varphi(y)]$ for every $x\in X$. Then $\varphi^d$ is $b$-Lipschitz function and $\varphi^d(x)\in [-a,a]$ for every $x\in X$. Therefore $\varphi^d\in \F$. Now we define $\varphi^{dd}(y):=\inf_{x\in X}[b\cdot d(x,y)-\varphi^d(x)]$ for every $y\in X$. Then $\varphi^{dd}$ is $b$-Lipschitz and 
$$\varphi^d(x)+\varphi^{dd}(y)\leq b\cdot d(x,y), \mbox{ for every } x,y\in X.$$
As $-a\leq \varphi^{d}(x)\leq a$ we also get that $-a\leq \varphi^{dd}(y)\leq a$ for every $y\in X$. Therefore we have $\varphi^{dd}\in \F$ and $(\varphi^d,\varphi^{dd})\in \Phi_W$.

On the other hand, as $\psi(x)+\varphi(y)\leq b\cdot d(x,y)$ for every $x,y\in X$ we get that 
$$\psi(x)\leq \inf_{y\in X}[b\cdot d(x,y)-\varphi(y)]=\varphi^d(x) \mbox{ for every } x\in X.$$
Similarly, from the definitions of $\varphi^{dd}$ we also have $\varphi^{dd}(y)\geq \varphi(y)$ for every $y\in Y$. Hence 
\begin{align*}
\int_{X} I\left(\psi\right)d\mu+\int_{X} I\left(\varphi\right)d\nu \leq \int_{X} I\left(\varphi^d\right)d\mu+\int_{X} I\left(\varphi^{dd}\right)d\nu.
\end{align*}
Therefore, $$\sup_{(\psi,\varphi)\in \Phi_W}\bigg\{\int_{X} I\left(\psi\right)d\mu+\int_{X} I\left(\varphi\right)d\nu \bigg\}\leq \sup_{\varphi\in C_b(X)}\bigg\{\int_{X} I\left(\varphi^d\right)d\mu+\int_{X} I\left(\varphi^{dd}\right)d\nu \bigg\}.$$
As $\varphi^d$ is $b$-Lipschitz we get $$-\varphi^d(x)\leq \inf_{y\in X}[b\cdot d(x,y)-\varphi^d(y)].$$
On the other hand, $\inf_{y\in X}[b\cdot d(x,y)-\varphi^d(y)]\leq -\varphi^d(x)$. Hence $$\varphi^{dd}(x)=\inf_{y\in X}[b\cdot d(x,y)-\varphi^d(y)]=-\varphi^d(x).$$
Thus
\begin{eqnarray*}
\sup_{(\psi,\varphi)\in \Phi_W}\bigg\{\int_{X} I\left(\psi\right)d\mu+\int_{X} I\left(\varphi\right)d\nu \bigg\}&\leq& \sup_{\varphi\in C_b(X)}\bigg\{\int_{X} I\left(\varphi^d\right)d\mu+\int_{X} I\left(\varphi^{dd}\right)d\nu \bigg\}\\
&= &\sup_{\varphi\in C_b(X)}\bigg\{\int_{X} I\left(\varphi^d\right)d\mu+\int_{X} I\left(-\varphi^{d}\right)d\nu \bigg\}\\
&\leq &\sup_{\varphi\in \F}\bigg\{\int_{X} I\left(\varphi\right)d\mu+\int_{X} I\left(-\varphi\right)d\nu \bigg\}\\
&\leq& \sup_{(\psi,\varphi)\in \Phi_W}\bigg\{\int_{X} I\left(\psi\right)d\mu+\int_{X} I\left(\varphi\right)d\nu \bigg\}.
\end{eqnarray*}
So we must have equality everywhere and get the result.
\end{proof}

\begin{remark}
1) Corollary \ref{C-flat metrics} (2) has been proved in \cite[Theorem 2]{PR16} for the case $a=b=1$ and $X=\R^n$ by a different method.

2) (\cite{Hanin1,Hanin2}) Let $(X,d)$ be a Polish metric space. Let $\cM^0(X)$ be the set of all $\mu\in \cM_s(X)$ such that $\mu(X)=0$. For every $\mu\in \cM^0(X)$ we denote by $\Psi_\mu$ the set of all nonnegative measures $\boldsymbol{\gamma}\in \cM(X\times X)$ such that $\lambda(X\times A)-\lambda(A\times X)=\mu(A)$ for every Borel $A\subset X$. Then we define for every $\mu\in \cM^0(X)$,
$$\|\mu\|^0_d:=\inf_{\boldsymbol{\gamma}\in \Psi_\mu}\bigg\{\int_{X\times X}d(x,y)d\boldsymbol{\gamma}(x,y)\bigg\}.$$
Now, on the vector space $\cM_s(X)$ we define an extension Kantorovich-Rubinstein norm as following
$$\|\mu\|_d:=\inf_{\nu\in \cM^0(X)}\bigg\{\|\nu\|^0_d+|\mu-\nu|(X)\bigg\}, \mbox{ for every } \mu\in \cM_s(X).$$
Then from \cite[Theorem 0]{Hanin1} (when $X$ is compact) or \cite[Theorem 1]{Hanin2} (when $X$ is a general Polish metric space), applying Hahn-Banach Theorem we get that 
$$\|\mu\|_d=\sup\bigg\{\int_X fd(\mu-\nu):f\in \F\bigg\},$$
where $\F:=\big\{f\in C_b(X), \|f\|_\infty\leq 1, \|f\|_{Lip}\leq 1\big\}$.
We thank Benedetto Piccoli and Francesco Rossi for pointing \cite{Hanin1} out to us, and we have found \cite{Hanin2} after that.
\end{remark}

Using Corollary \ref{C-flat metrics} (2) we get another proof of \cite[Lemma 5]{PRT}.
\begin{cor} Let $X$ be a locally compact, Polish metric space. For every $\mu,\nu,\eta\in \cM(X)$ we have 
$$\widetilde{W}_1^{a,b}(\mu+\eta,\nu+\eta)=\widetilde{W}^{a,b}_1(\mu,\nu).$$
\end{cor}
From Corollary \ref{C-flat metrics} (2) we also get a similar result in \cite[Lemma 1.5]{Basso}.
\begin{cor}
Let $\left(X_1,d_1\right)$ and $\left(X_2,d_2\right)$ be locally compact, Polish metric spaces. If $\psi:\left(X_1,d_1\right)\rightarrow\left(X_2,d_2\right)$ is an isometry map then the map $\psi_\sharp:\left(\mathcal{M}\left(X_1\right), \widetilde{W}^{a,b}_1\right)\rightarrow \left(\mathcal{M}\left(X_2\right), \widetilde{W}^{a,b}_1\right)$ is also an isometry.
\end{cor}
\begin{proof} For every $\mu,\nu\in \cM(X_1)$, it is clear that $\widetilde{W}^{a,b}_1\left(\mu,\nu\right)\geq \widetilde{W}^{a,b}_1\left(\psi_\sharp\mu,\psi_\sharp\nu\right)$ and $\psi_\sharp$ is surjective. Hence, we need to show that $\widetilde{W}^{a,b}_1\left(\psi_\sharp\mu,\psi_\sharp\nu\right)\geq \widetilde{W}^{a,b}_1\left(\mu,\nu\right)$.\\
Let $\F_i =\left\lbrace f\in C_b(X_i), \Vert f\Vert_\infty \leq a, \Vert f\Vert_{Lip}\leq b\right\rbrace, i=1,2$. By Corollary \ref{C-flat metrics} (2) one has
\begin{align*}
\widetilde{W}_1^{a,b} \left(\psi_{\sharp}\mu, \psi_{\sharp}\nu\right)=\sup\limits_{g\in \F_2}\int_{X_2} g\,d\left(\psi_{\sharp}\mu- \psi_{\sharp}\nu\right) = \sup\limits_{g\in \F_2} \int_{X_1} g\circ\psi \,d\left(\mu-\nu\right).
\end{align*}
For every $f\in \F_1$ and every $y\in X_2$ we define $h(y):= \inf\limits_{x\in X_1} \left[b\cdot d_2\left(y,\psi (x)\right) + f(x) \right]$. Then $h$ is $b$-Lipschitz and $h(y)\geq -a$ for every $y\in X_2$. Since $\psi$ is surjective, for every $y\in X_2$, there exists $x'\in X_1$ such that $\psi(x')=y$. Thus, $h(y)\leq f(x')\leq a$. Therefore, $h\in \F_2$. Moreover, since $f$ is $b$-Lipschitz, for every $x\in X_1$ one has 
\begin{align*}
f(x)&=\inf\limits_{x_1\in X_1} \left[ f(x_1) + \vert f(x) -f(x_1)\vert \right]\\ &\leq \inf\limits_{x_1\in X_1} \left[ f(x_1) + b\cdot d_1(x,x_1) \right]\\ 
&= \inf\limits_{x_1\in X_1} \left[ f(x_1) + b\cdot d_2\left(\psi(x) ,\psi( x_1)\right) \right]\\&\leq f(x).
\end{align*}
Therefore, $f(x)=h\left(\psi (x)\right)$, for all $x\in X_1$ or $f=h\circ \psi$. Hence, we get that 
\begin{align*}
\int_{X_1} f\, d(\mu-\nu) =\int_{X_1} h\circ \psi \, d(\mu-\nu) \leq \sup\limits_{g\in \F_2}\int_{X_1} g\circ\psi \,d\left(\mu-\nu\right)= \widetilde{W}_1^{a,b} \left(\psi_{\sharp}\mu, \psi_{\sharp}\nu\right).
\end{align*}
So that $\widetilde{W}_1^{a,b} \left(\mu, \nu\right)\leq \widetilde{W}_1^{a,b} \left(\psi_{\sharp}\mu, \psi_{\sharp}\nu\right)$.
\end{proof} 

\section{Barycenter problem and an its dual problem}
Let $(X,d)$ be a locally compact, Polish metric space. For every integer $k\geq 2$, we consider $k$ measures $\mu_1,\mu_2,\ldots,\mu_k$ in $\mathcal{M}(X)$ such that $\text{supp}(\mu_i)$ is a compact subset of $X$ for every $i\in\{1,\ldots,k\}$.
Let $\lambda_1,\lambda_2,\ldots,\lambda_k$ be positive real numbers such that $\sum_{i=1}^k\lambda_i=1$ and let $K=\bigcup_{i=1}^k \text{supp}(\mu_i)$, we consider the following problem $$(B) \inf_{\text{supp}(\mu)\subset K} \sum_{i=1}^k\lambda_i \widetilde{W}^{a,b}_2\left(\mu_i,\mu\right)^2.$$  
\begin{remark} Let $X=\R^d$. For every $m>0, a,b\geq 0$ and $\mu_1,\mu_2\in \cM(X)$ we define $$\widetilde{W}_{2,m}^{a,b}(\mu_1,\mu_2):=\inf_{\gamma_i\in \cM_2(X), \gamma_i\leq \mu_i,\vert \boldsymbol{\gamma}\vert=m}a\sum_{i=1}^2\vert \mu_i-\gamma_i\vert+b\int_{X\times X} \vert x-y\vert^2 d\boldsymbol{\gamma}(x,y).$$
In \cite{KP} Kitagawa and Pass introduced and investigated the following partial barycenter problem.
$$\inf_{\mu\in\cM(X),\vert \mu\vert=m}\sum_{i=1}^k\widetilde{W}_{2,m}^{0,1}(\mu_i,\mu)^2.$$
The methods there are different from us as they study their partial barycenters via multi-marginal optimal transports while we use duality formulations for our barycenter problems in generalized Wasserstein spaces.
\end{remark}

\begin{thm}\label{T-existence of barycenter}
Problem $(B)$ has solutions.
\end{thm}
\begin{proof}
For every $\mu\in\cM(X)$ such that $\text{supp}(\mu)\subset K$, let $J(\mu)=\sum_{i=1}^k\lambda_i \widetilde{W}^{a,b}_2\left(\mu_i,\mu\right)^2.$ Let $\left\{\mu^n\right\}_{n\in\mathbb{N}}$ be a minimizing sequence of $(B)$. For every $n\in\mathbb{N}$, let $x\not\in \text{supp}(\mu^n)$ then there exists an open neighborhood $U_x$ of $x$ such that $\mu^n\left(U_x\right)=0$. Since $X$ is separable and $\left\{U_x\right\}_{x\in X\backslash\text{supp}(\mu^n)}$ is an open cover of $X\minus\text{supp}(\mu^n)$, applying Lindel\"{o}f Theorem there is a countable subcover $\left\{U_{x_i}\right\}_i$. Therefore, $\mu^n\left(X\minus\text{supp}(\mu^n)\right)=0$. Moreover, $\text{supp}(\mu^n)\subset K$ for every $n\in \mathbb{N}$. Thus, for every $n\in N$, $\mu^n(X\minus K)=0$. It implies that $\left\{\mu^n\right\}_{n\in\mathbb{N}}$ is tight.

We now prove that $\left\{\mu^n\right\}_{n\in\mathbb{N}}$ is bounded. For every $n\in\mathbb{N}$ and every $i\in\{1,2,\ldots,k\}$, using Corollary \ref{C-flat metrics} (1) we get that \begin{align*}
\widetilde{W}^{a,b}_2\left(\mu^n,\mu_i\right)^2=\sup\bigg\{\int_X\varphi_1(x) d\mu^n(x)+\int_X\varphi_2(x) d\mu_i(x)|\left(\varphi_1,\varphi_2\right)\in\Phi_W\bigg\},
\end{align*}
We set $\varphi_1(x)=a, \varphi_2(x)=-a$ for every $x\in X$ then $$\lambda_i.\widetilde{W}^{a,b}_2\left(\mu^n,\mu_i\right)^2\geq \lambda_ia\mu^n(X)-\lambda_ia\mu_i(X)$$
This yields, $$\left\vert \mu^n\right\vert\leq \dfrac{1}{a}J\left(\mu^n\right)+\sum_i^k\lambda_i\vert \mu_i\vert,\text{ for every }n\in\mathbb{N}.$$
As $\mu_i\in\cM(X)$ for every $i\in \{1,2,\ldots,k\}$ and $J\left(\mu^n\right)$ is bounded, we obtain that $\left\{\mu^n\right\}_{n\in\mathbb{N}}$ is bounded. Therefore, applying Prokhorov's Theorem, passing to a subsequence we can assume that $\mu^n\rightarrow \mu$ as $n\rightarrow\infty$ in the weak*-topology for some $\mu\in\cM(X)$.

We now show that $\text{supp}(\mu)\subset K$. As $X\minus K$ is an open set, applying \cite[Theorem 6.1]{Par} we get that $$0=\liminf_{n\rightarrow \infty}\mu^n(X\minus K)\geq \mu(X\minus K).$$ 
Therefore $X\minus K\subset X\minus\text{supp}(\mu)$. Hence, $\text{supp}(\mu)\subset K$. 

Next, we will check that $\widetilde{W}^{a,b}_2(\mu^n,\mu)\rightarrow 0$ as $n\rightarrow\infty$. If $\vert \mu\vert=0$ then we are done. If $\vert\mu\vert>0$ then there exists $N>0$ such that $\vert\mu^n\vert>0$ for all $n\geq N$. For each $n\geq N$, we define $\nu^n:=\vert\mu\vert\mu^n/\vert \mu^n\vert$ then $\vert \nu^n\vert=\vert\mu\vert$. Therefore, \begin{align*}
\widetilde{W}^{a,b}_2(\mu^n,\mu)^2\leq a\vert\mu^n-\nu^n \vert+b^2W_2^2(\nu^n,\mu)=a\left\vert\vert\mu^n\vert-\vert \mu\vert\right\vert+b^2W_2^2(\nu^n,\mu).
\end{align*}
Moreover, since $\mu^n\rightarrow \mu$ as $n\rightarrow\infty$ one has $\nu^n\rightarrow\mu$ as $n\rightarrow\infty$. Observe that $\nu^n$ and $\mu$ are concentrated in 
compact set $K$, applying \cite[Definition 6.8 and Theorem 6.9]{V09} we obtain that $\lim_{n\rightarrow \infty}W_2(\nu^n,\mu)=0$. This yields, $$\limsup_{n\rightarrow \infty}\widetilde{W}^{a,b}_2(\mu^n,\mu)^2\leq a\lim_{n\rightarrow\infty}\left\vert\vert\mu^n\vert-\vert \mu\vert\right\vert+\lim_{n\rightarrow\infty}b^2W_2^2(\mu^n,\mu)=0.$$
Notice that $\liminf_{n\rightarrow\infty}\widetilde{W}^{a,b}_2(\mu^n,\mu)\geq 0$. Therefore, $\lim_{n\rightarrow\infty}\widetilde{W}^{a,b}_2(\mu^n,\mu)=0$. This implies that $\lim_{n\rightarrow\infty}J(\mu^n)=J(\mu)$. Hence, we get the result.
\end{proof}
\begin{definition}
Let $X$ be a locally compact, Polish metric space. For every integer $k\geq 2$, let $\mu_1,\ldots,\mu_2\in \mathcal{M}(X)$ such that $\text{supp}(\mu_i)$ is a compact subset of $X$, for every $i\in\{1,\ldots,k\}$. Let $k$ positive real numbers $\lambda_1,\ldots,\lambda_k$ such that $\sum_{i=1}^k\lambda_i=1$. We say that $\mu\in \cM(X)$ is a generalized Wasserstein barycenter of $\left(\mu_1,\ldots,\mu_k\right)$ with weights $\left(\lambda_1,\ldots,\lambda_k\right)$ if $\mu$ is a solution of $(B)$. We denote by $BC\left((\mu_i,\lambda_i)_{1\leq i\leq k}\right)$ the set of all generalized Wasserstein barycenters of $\left(\mu_1,\ldots,\mu_k\right)$ with weights $\left(\lambda_1,\ldots,\lambda_k\right)$.
\end{definition}

In general, barycenters in a generalized Wasserstein space are not unique. 
\begin{example}
Let $X=\mathbb{R},a=b=1$ and $\lambda_1=\lambda_2=1/2$. For every $x\geq 0$ let $\mu_1=\delta_x$ and $\mu_2=3\delta_x$. Then we have $\left\{\mu\in\cM(\mathbb{R})|\text{supp}(\mu)\subset\{x\}\right\}=\{q\delta_x|q\geq 0\}$. For every $q\geq 0$, let $(\widetilde{\mu}_1,\widetilde{\mu}_2)$ be an optimal for $\widetilde{W}^{1,1}_2(\delta_x,q\delta_x)$. Since $\vert\widetilde{\mu}_1\vert=\vert \widetilde{\mu}_2\vert, \widetilde{\mu}_1\leq \delta_x,\widetilde{\mu}_2\leq q\delta_x$, we must have $\widetilde{\mu}_1=\widetilde{\mu}_2=r\delta_x$ where $0\leq r\leq \min\{q,1\}$. Hence, we get that $$\widetilde{W}^{1,1}_2(\delta_x,q\delta_x)^2=\min\{q+1-2r|0\leq r\leq \min\{q,1\}\}.$$
Similarly, we also get that $$\widetilde{W}^{1,1}_2(3\delta_x,q\delta_x)^2=\min\{q+3-2s|0\leq s\leq \min\{q,3\}\}.$$
It is easy to check that $$\lambda_1.\min\{q+1-2r|0\leq r\leq \min\{q,1\}\}+\lambda_2.\min\{q+3-2s|0\leq s\leq \min\{q,3\}\}=1,$$
and the minimum is attained when $q\in[1,3]$. Therefore, $BC((\mu_1,\lambda_1),(\mu_2,\lambda_2))=\{q\delta_x|q\in[1,3]\}.$
\end{example}

We now prove the consistency of barycenters in generalized Wasserstein spaces which has been shown in \cite[Theorem 3.1]{Boissard} for the Wasserstein setting.
\begin{thm}\label{T-consistency}  
Let $(X,d)$ be a locally compact, Polish metric space. For every integer $k\geq 2$, let $\left\{\mu^n_i\right\}\subset \cM(X)$ be sequences converging in generalized Wasserstein distance to compactly supported measure $\mu_i\in \cM(X)$ for every  $i\in\{1,\ldots,k\}$. Let $K=\bigcup_{i=1}^k\text{supp}(\mu_i)$ and let $k$ positive integers $\lambda_1,\ldots,\lambda_k$ such that $\sum_{i=1}^k\lambda_i=1$. For each $n\in\mathbb{N}$, suppose that $\text{supp}(\mu_i^n)\subset K$ for every $i\in\{1,\ldots,k\}$. Then $BC\left((\mu_i^n,\lambda_i)_{1\leq i\leq k}\right)$ is a nonempty set for every $n\in\N$. Moreover, for every $n\in\N$, let $\mu_B^n\in BC\left((\mu_i^n,\lambda_i)_{1\leq i\leq k}\right)$ then the sequence $\left\{\mu^n_B\right\}$ is precompact in $(\cM(X),\widetilde{W}^{a,b}_2)$ and any its limit point is a generalized Wasserstein barycenter of $\left(\mu_1,\ldots,\mu_k\right)$ with weights $\left(\lambda_1,\ldots,\lambda_k\right)$. 
\end{thm}
\begin{proof} Since $\text{supp}(\mu_i^n)\subset K$ and $K$ is compact, one has $\text{supp}(\mu_i^n)$ is compact for every $n\in\mathbb{N}$ and every $i\in\{1,\ldots,k\}$. Therefore, $BC\left((\mu_i^n,\lambda_i)_{1\leq i\leq k}\right)$ is a nonempty set for every $n\in \mathbb{N}$, this follows from Theorem \ref{T-existence of barycenter}. 

We now prove the second part. Since $\mu_B^n\in BC\left((\mu_i^n,\lambda_i)_{1\leq i\leq k}\right)$, we get that $\text{supp}(\mu_B^n)\subset \bigcup_{i=1}^k\text{supp}(\mu_i^n)\subset K$,  for every $n\in \N$. Then $\mu_B^n\left(X\minus K\right)=0$ for every $n\in \N$. Therefore, $\left\{\mu_B^n\right\}$ is tight. Let $\mu_B\in BC\left((\mu_i,\lambda_i)_{1\leq i\leq k}\right)$. Since $\widetilde{W}^{a,b}_2\left(\mu_i^n,\mu_i\right)\rightarrow 0$ as $n\rightarrow \infty$ for every $i\in\{1,\ldots,k\}$ we get that $$\lim_{n\rightarrow\infty}\widetilde{W}^{a,b}_2\left(\mu_B,\mu_i^n\right)=\widetilde{W}^{a,b}_2\left(\mu_B,\mu_i\right)<\infty$$
Therefore, $\{\widetilde{W}^{a,b}_2\left(\mu_B,\mu_i^n\right)\}_n$ is bounded for every $i\in\{1,\ldots,k\}$. Moreover, \begin{align}\label{T-consistency inequality}
\sum_{i=1}^k\lambda_i \widetilde{W}^{a,b}_2\left(\mu_B^n,\mu_i^n\right)^2\leq \sum_{i=1}^k\lambda_i \widetilde{W}^{a,b}_2\left(\mu_B,\mu_i^n\right)^2,\text{ for every }n\in \N.
\end{align}
This yields, $\widetilde{W}^{a,b}_2\left(\mu_B^n,\mu_i^n\right)$ is bounded for every $i\in\{1,\ldots,k\}$. As $\mu^n_i\rightarrow \mu_i$ as $n\rightarrow\infty$ in the weak*-topology, applying \cite[Theorem 6.1]{Par} we get that $\lim_{n\rightarrow\infty}\mu_i^n(X)=\mu_i(X)<\infty$. Thus, $\left\{\mu_i^n\right\}$ is bounded for every $i\in\{1,\ldots,k\}$. Therefore, using Corollary \ref{C-flat metrics} (1) and by the same arguments as in the proof of Theorem \ref{T-existence of barycenter} we obtain that $\left\{\mu_B^n\right\}$ is bounded. Hence, applying Prokhorov's Theorem, passing to a subsequence we can assume that $\mu_B^n\rightarrow \widehat{\mu}_B$ as $n\rightarrow \infty$ in the weak*-topology for some $\widehat{\mu}_B\in\cM(X)$. Observe that, from $\mu^n_B\left(X\minus K\right)=0$ for every $n\in\N$ and $X\minus K$ is an open set, we get that $\widehat{\mu}_B(X\minus K)=0$ and thus $\text{supp}(\widehat{\mu}_B)\subset K$. By the same arguments in the proof of Theorem \ref{T-existence of barycenter} we also have $\widetilde{W}^{a,b}_2\left(\mu_B^n,\widehat{\mu}_B\right)\rightarrow 0$ as $n\rightarrow\infty$. This implies that the sequence $\left\{\mu^n_B\right\}$ is precompact in generalized Wasserstein topology and we also get that $$\lim_{n\rightarrow \infty}\widetilde{W}^{a,b}_2\left(\mu_B^n,\mu_i^n\right)=\widetilde{W}^{a,b}_2\left(\widehat{\mu}_B,\mu_i\right),\text{ for every }i\in\{1,\ldots,k\}.$$    
Hence, since (\ref{T-consistency inequality}) we get that \begin{align*}
\sum_{i=1}^k\lambda_i \widetilde{W}^{a,b}_2\left(\widehat{\mu}_B,\mu_i\right)^2&=\lim_{n\rightarrow\infty}\sum_{i=1}^k\lambda_i \widetilde{W}^{a,b}_2\left(\mu_B^n,\mu_i^n\right)^2\\
&\leq \lim_{n\rightarrow\infty}\sum_{i=1}^k\lambda_i \widetilde{W}^{a,b}_2\left(\mu_B,\mu_i^n\right)^2\\
&= \sum_{i=1}^k\lambda_i \widetilde{W}^{a,b}_2\left(\mu_B,\mu_i\right)^2.
\end{align*}
Therefore, $\widehat{\mu}_B\in BC\left((\mu_i,\lambda_i)_{1\leq i\leq k}\right)$. 
\end{proof}

Next, we will study the a dual problem of problem $(B)$. For every $\lambda>0$ and every function $f\in C_b(K)$ such that $f(x)\leq \lambda a$ for every $x\in K$ where $K=\bigcup _{i=1}^k\text{supp}(\mu_i)$, we define 
 $S_\lambda f(x):=\inf_{y\in K}\left\{\lambda b^2 d^2(x,y)-f(y)\right\}$ and $\overline{S}_\lambda f(x):=\min\left\{S_\lambda f(x),\lambda a\right\}$. For every integer $k\geq 2$ and for each $i\in\{1,2,\ldots,k\}$ we define function $H_i:C_b(K)\rightarrow \overline{\mathbb{R}}$ by 
\[H_i(f):=\left\{\begin{array}{l}-\int_K \overline{S}_{\lambda_i}f(x)d\mu_i(x) \text{ if } f\in F_{\lambda_i}\\+\infty \text{ otherwise, }\end{array}\right.\] where $F_{\lambda_i}:=\left\{f\in C_b(K)| f(x)\leq \lambda_ia\text{ for every }x\in K\right\}$. Then $H_i$ is a convex function on $F_{\lambda_i}$.   

We denote by $\cM_s(K)$ (resp. $\cM_c(K))$ the space of signed (resp. nonnegative) Radon measures $\mu$ with finite mass on $X$ such that $\mu$ is concentrated on $K$, i.e. $\mu(X\minus K)=0$. Then $\cM_s(K)$ is the dual space of $C_b(K)$, since $K$ is compact. For every $\mu\in \cM_s(K)$, the Legendre-Fenchel transform of $H_i$ is \begin{align*}
H_i^*(\mu)&=\sup\left\{\int_K f(x)d\mu(x)-H_i(f)| f\in C_b(K)\right\}\\
&=\sup\left\{\int_K f(x)d\mu(x)-H_i(f)| f\in F_{\lambda_i}\right\}\\
&=\sup\left\{\int_K f(x)d\mu(x)+\int_K\overline{S}_{\lambda_i}f(x)d\mu_i(x)| f\in F_{\lambda_i}\right\}.
\end{align*} 

We consider the following problem $$(B^*)\;\sup\left\{ \sum_{i=1}^k\int_K\overline{S}_{\lambda_i}f_i(x)d\mu_i(x)|f_i\in F_{\lambda_i},\sum_{i=1}^kf_i=0\right\}. $$
\begin{lem}\label{L-inf B>sup B*} Let $X$ be a locally compact, Polish metric space then $\inf(B)\geq \sup (B^*)$.
\end{lem}
\begin{proof} For $i=1,2,\ldots,k$ let any $f_i\in F_{\lambda_i}$ such that $\sum_{i=1}^kf_i=0$. Then $\overline{S}_{\lambda_i}f_i(x)+f_i(y)\leq \lambda_ib^2d^2(x,y)$ for every $x,y\in K$ and every $i\in \{1,2,\ldots,k\}$. For every $\mu\in\cM_c(K)$, let $\boldsymbol{\gamma}^i\in M^\leq \left(\mu,\mu_i\right)$ be an optimal plan for $\widetilde{W}^{a,b}_2\left(\mu,\mu_i\right)$. Since $\mu_i$ is concentrated on $K$ for every $i=1,\ldots,k$, we get that $$\widetilde{W}^{a,b}_2\left(\mu,\mu_i\right)^2=a\left( \mu-\pi_\sharp^1\boldsymbol{\gamma}^i\right)(K)+a\left( \mu_i-\pi_\sharp^2\boldsymbol{\gamma}^i\right)(K)+b^2\int_{K\times K}d^2(x,y)d\boldsymbol{\gamma}^i(x,y).$$
As $\boldsymbol{\gamma}^i\in M^\leq \left(\mu,\mu_i\right)$, by Radon-Nikodym Theorem there exist measurable functions $\varphi_1,\varphi_2:K\rightarrow [0,+\infty)$ such that $\pi_\sharp^1\boldsymbol{\gamma}^i=\varphi_1\mu$, $\pi_\sharp^2\boldsymbol{\gamma}^i=\varphi_2\mu_i$ and $\varphi_1\leq 1 \;\mu$-a.e, $\varphi_2\leq 1 \;\mu_i$-a.e. Therefore, we get that \begin{align*}
\widetilde{W}^{a,b}_2\left(\mu,\mu_i\right)^2 &= a\int_K\left(1-\varphi_1\right)d\mu+a\int_K\left(1-\varphi_2\right)d\mu_i+b^2\int_{K\times K}d^2(x,y)d\boldsymbol{\gamma}^i(x,y)\\
& \geq a\int_K\left(1-\varphi_1\right)d\mu+a\int_K\left(1-\varphi_2\right)d\mu_i+\dfrac{1}{\lambda_i}\int_{K\times K}\left[f_i(x)+\overline{S}_{\lambda_i}f_i(y)\right]d\boldsymbol{\gamma}^i(x,y)\\
& = \int_K\left[a\left(1-\varphi_1\right)+\dfrac{1}{\lambda_i}f_i.\varphi_1\right]d\mu+\int_K\left[a\left(1-\varphi_2\right)+\dfrac{1}{\lambda_i}\overline{S}_{\lambda_i}f_i.\varphi_2\right]d\mu_i.
\end{align*}
Moreover, $\varphi_1(x),\varphi_2(x)\geq 0$ for every $x\in X$ and $\varphi_1\leq 1 \;\mu$-a.e, $\varphi_2\leq 1 \;\mu_i$-a.e, $f_i(x)/\lambda_i\leq a$, $\overline{S}_{\lambda_i}f_i(x)/\lambda_i\leq a$ for every $x\in K$. Therefore, we obtain that \begin{align*}
a\left(1-\varphi_1(x)\right)+\left(f_i(x)/\lambda_i\right).\varphi_i(x)\geq f_i(x)/\lambda_i,\;\mu-\text{a.e}, \\ 
a\left(1-\varphi_2(x)\right)+\left(\overline{S}_{\lambda_i}f_i(x)/\lambda_i\right).\varphi_2(x)\geq \overline{S}_{\lambda_i}f_i(x)/\lambda_i,\;\mu_i-\text{a.e}.
\end{align*}  
Hence, for every $i\in \{1,2,\ldots,k\}$, we get that 
\begin{align}\label{L-inf B>sup B* inequality}
\lambda_i\widetilde{W}^{a,b}_2\left(\mu,\mu_i\right)^2\geq \int_K f_i(x)d\mu(x)+\int_K\overline{S}_{\lambda_i}f_i(x)d\mu_i(x).
\end{align}
Thus, \begin{align*}
\sum_{i=1}^k\lambda_i\widetilde{W}^{a,b}_2\left(\mu,\mu_i\right)^2&\geq \sum_{i=1}^k\int_K f_i(x)d\mu(x)+\sum_{i=1}^k\int_K\overline{S}_{\lambda_i}f_i(x)d\mu_i(x)\\
&=\sum_{i=1}^k\int_K\overline{S}_{\lambda_i}f_i(x)d\mu_i(x).
\end{align*}
This yields,
$$\inf\left\{\sum_{i=1}^k\lambda_i\widetilde{W}^{a,b}_2\left(\mu,\mu_i\right)^2|\text{ supp}(\mu)\subset K\right\}\geq \sum_{i=1}^k\int_K\overline{S}_{\lambda_i}f_i(x)d\mu_i(x).$$
Hence, we get the result.
\end{proof}
\begin{lem}\label{L-dual functional}
Let $X$ be a locally compact, Polish metric space. Then for every $i\in\{1,2,\ldots,k\}$ we have $H_i^*(\mu)=\lambda_i\widetilde{W}^{a,b}_2\left(\mu,\mu_i\right)^2$ if $\mu\in\cM_c(K)$ and $+\infty$ otherwise.
\end{lem}
\begin{proof} If $\mu\in \cM_s(K)\minus \cM_c(K)$ then there exists $g\in C_b(K),g\leq 0$ such that $\int_K g(x)d\mu(x)>0$. For every $t\in \R,t\geq 0$ let $f=t.g$ then $f\in F_{\lambda_i}$ and $\overline{S}_{\lambda_i}(tf(x))\geq 0$ for every $x\in K$. Therefore, $$H_i^*(\mu)\geq \sup_{t\geq 0}\int_K fd\mu=+\infty.$$

We now consider $\mu\in \cM_c(K)$. Since (\ref{L-inf B>sup B* inequality}), it is clear that $\lambda_i\widetilde{W}^{a,b}_2\left(\mu,\mu_i\right)^2\geq H_i^*(\mu)$. So we need to prove that $\lambda_i\widetilde{W}^{a,b}_2\left(\mu,\mu_i\right)^2\leq H_i^*(\mu)$. We define
\begin{align*}
\Phi_K:=\bigg\{(\varphi_1,\varphi_2)\in C_b(K)\times C_b(K)&:\varphi_1(x)+\varphi_2(y)\leq b^2d^2(x,y), \varphi_1(x),\varphi_2(y)\geq -a,\\
&\text{ for every } x,y\in K\bigg\}.
\end{align*}
Let any $\left(\varphi_1,\varphi_2\right)\in \Phi_K$ then $\lambda_i\varphi_1(x)+\lambda_i\varphi_2(y)\leq \lambda_ib^2d^2(x,y)$ for every $x,y\in K$ and every $i=1,\ldots,k$. Therefore, $\lambda_i\varphi_2(y)\leq S_{\lambda_i}\left(\lambda_i\varphi_1(y)\right)$ for every $y\in K$. Observe that $\varphi_2(y)\in [-a,a]$ for every $y\in K$, we get that $\lambda_i\varphi_2(y)\leq \overline{S}_{\lambda_i}\left(\lambda_i\varphi_1(y)\right)$ for every $y\in K$. As $\lambda_i\varphi_1(x)\leq \lambda_ia$ for every $x\in K$, one has $\lambda_i\varphi_1\in F_{\lambda_i}$. Hence, we obtain that 
\begin{align*}
\int_K\lambda_i\varphi_1(x)d\mu(x)+\int_K\lambda_i\varphi_2(y)d\mu_i(y)&\leq \int_K\lambda_i\varphi_1(x)d\mu(x)+\int_K\overline{S}_{\lambda_i}\left(\lambda_i\varphi_1(y)\right)d\mu_i(y)\\
&\leq H_i^*(\mu).
\end{align*}
Applying Corollary \ref{C-flat metrics} (1) we get that $$\widetilde{W}^{a,b}_2\left(\mu,\mu_i\right)^2=\sup_{\left(\varphi_1,\varphi_2\right)\in \Phi_K}\left\{\int_K\varphi_1(x)d\mu(x)+\int_K\varphi_2(y)d\mu_i(y)\right\}\leq \dfrac{1}{\lambda_i}H_i^*(\mu).$$
Hence, $\lambda_i\widetilde{W}^{a,b}_2\left(\mu,\mu_i\right)^2\leq H_i^*(\mu)$ for every $\mu\in \cM_c(K)$ and every $i\in\{1,2,\ldots,k\}$.
\end{proof}

Let $F:=\left\{f\in C_b(K)|f(x)\leq a\text{ for every }x\in K\right\}$. We define $H:C_b(K)\rightarrow\overline{\mathbb{R}}$ by $H(f)=\inf\left\{\sum_{i=1}^kH_i\left(f_i\right)|f_i\in F_{\lambda_i},\sum_{i=1}^kf_i=f\right\}$ if $f\in F$ and $+\infty$ otherwise.

\begin{lem}\label{L-sum of dual functions}
$H$ is convex on $F$ and $H^*(\mu)=\sum_{i=1}^kH_i^*(\mu)$ for every $\mu\in M_s(K)$.
\end{lem}
\begin{proof}
For every $g_1,g_2\in F$ and every $t\in [0,1]$ we will check that $H\left(tg_1+(1-t)g_2\right)\leq tH(g_1)+(1-t)H(g_2)$. Let any $\overline{f}_i,\widehat{f}_i\in F_{\lambda_i}$ such that $\sum_{i=1}^k\overline{f}_i=g_1$ and $\sum_{i=1}^k\widehat{f}_i=g_2$ then $t\overline{f}_i+(1-t)\widehat{f}_i\in F_{\lambda_i}$ and $\sum_{i=1}^k\left[t\overline{f}_i+(1-t)\widehat{f}_i\right]=tg_1+(1-t)g_2$. As $H_i$ is convex on $F_{\lambda_i}$ for every $i=1,\ldots,k$, we get that \begin{align*}
t\sum_{i=1}^kH_i\left(\overline{f}_i\right)+(1-t)\sum_{i=1}^kH_i\left(\widehat{f}_i\right)&=\sum_{i=1}^k\left[tH_i\left(\overline{f}_i\right)+(1-t)H_i\left(\widehat{f}_i\right)\right]\\
&\geq \sum_{i=1}^kH_i\left(t\overline{f}_i+(1-t)\widehat{f}_i\right)\\
&\geq H\left(tg_1+(1-t)g_2\right).
\end{align*}
Therefore, $H\left(tg_1+(1-t)g_2\right)\leq tH(g_1)+(1-t)H(g_2)$. It implies that $H$ is convex on $F$.

We now show that $H^*(\mu)=\sum_{i=1}^kH_i^*(\mu)$ for every $\mu\in M_s(K)$. For every $\mu\in \cM_s(K)$, by definition of the Legendre-Fenchel one has \begin{align*}
H^*(\mu)&=\sup_{f\in C_b(K)}\left\{\int_Kfd\mu-H(f)\right\}\\
&=\sup_{f\in F}\left\{\int_Kfd\mu-H(f)\right\}\\
& = \sup_{f\in F}\left\{\int_Kfd\mu-\inf\left\{\sum_{i=1}^kH_i(f_i)|f_i\in F_{\lambda_i},\sum_{i=1}^kf_i=f\right\}\right\}\\
& = \sup_{f\in F}\left\{\int_Kfd\mu+\sup\left\{\sum_{i=1}^k\int_K\overline{S}_{\lambda_i}f_id\mu_i|f_i\in F_{\lambda_i},\sum_{i=1}^kf_i=f\right\}\right\}.
\end{align*} 
For every $f_i\in F_{\lambda_i}$ let $f=\sum_{i=1}^kf_i$ then $f\in F$. Therefore, for every $\mu\in \cM_s(K)$ we get that \begin{align*}
\sum_{i=1}^k\left(\int_Kf_i(x)d\mu(x)+\int_K\overline{S}_{\lambda_i}f_i(x)d\mu_i(x)\right)&=\int_Kf(x)d\mu(x)+\sum_{i=1}^k\int_K\overline{S}_{\lambda_i}f_i(x)d\mu_i(x)\\
&\leq H^*(\mu).
\end{align*}
This yields,\begin{align*}
\sum_{i=1}^kH_i^*(\mu)& = \sum_{i=1}^k\sup\left\{\int_Kf_i(x)d\mu(x)+\int_K\overline{S}_{\lambda_i}f_i(x)d\mu_i(x)|f_i\in F_{\lambda_i}\right\}\\
& = \sup\left\{\sum_{i=1}^k\left(\int_Kf_i(x)d\mu(x)+\int_K\overline{S}_{\lambda_i}f_i(x)d\mu_i(x)\right)|f_i\in F_{\lambda_i}\right\}\\
&\leq H^*(\mu).
\end{align*}
Conversely, for every $f\in F$ let $G:=\left\{(f_1,\ldots,f_k)|f_i\in F_{\lambda_i},\sum_{i=1}^kf_i=f\right\}$. Then we get that
\begin{align*}
\int_Kfd\mu+\sup_{(f_1,\ldots,f_k)\in G}\sum_{i=1}^k\int_k\overline{S}_{\lambda_i}f_id\mu_i& =\sup_{(f_1,\ldots,f_k)\in G}\left\{\int_Kfd\mu+\sum_{i=1}^k\int_k\overline{S}_{\lambda_i}f_id\mu_i\right\}\\
& = \sup_{(f_1,\ldots,f_k)\in G}\left\{\int_K\sum_{i=1}^kf_id\mu+\sum_{i=1}^k\int_k\overline{S}_{\lambda_i}f_id\mu_i\right\}\\
&\leq \sum_{i=1}^k\sup_{f_i\in F_{\lambda_i}}\left\{\int_Kf_id\mu+\int_k\overline{S}_{\lambda_i}f_id\mu_i\right\}\\
& = \sum_{i=1}^kH_i^*(\mu).
\end{align*}
Hence, we get the result.
\end{proof}

Inspired by \cite[Proposition 2.2]{Agueh} we get the following theorem.
\begin{thm}
Let $(X,d)$ be a locally compact, Polish metric space then $\inf(B)=\sup(B^*)$. 
\end{thm}
\begin{proof}
Combining Lemma \ref{L-dual functional} and Lemma \ref{L-sum of dual functions} we obtain that $$\inf(B)=\inf_{\mu\in\cM_c(K)}\sum_{i=1}^kH_i^*(\mu)=-\left(\sum_{i=1}^kH_i^*\right)^*(0)=-H^{**}(0).$$
Furthermore, we also have $\sup(B^*)=-H(0)$. Thus, we only need to prove that $H^{**}(0)=H(0)$. For every $f\in F$, let $f_i\in F_{\lambda_i}$ such that $\sum_{i=1}^kf_i=f$. As $f_i(x)\leq \lambda_ia$ for every $x\in K$ and every $i=1,\ldots,k$, one has $$S_{\lambda_i}f_i(x)=\inf_{y\in K}\left\{\lambda_ib^2d^2(x,y)-f_i(y)\right\}\geq -\lambda_ia \text{ for every }x\in K.$$ Therefore, $H_i\left(f_i\right)\leq \lambda_ia$. Moreover, since $\overline{S}_{\lambda_i}f(x)\leq \lambda_ia$ for every $x\in K$, we also have $H_i(f_i)\geq -\lambda_ia$. Hence $H$ is bounded on $F$. Thanks to Lemma \ref{L-sum of dual functions}, one has $H$ is convex on $F$. We denote by $\mathring{F}$ the interior of $F$, then $\mathring{F}$ is also a convex set. Applying \cite[Lemma 2.1]{Ekeland} we get that $H$ is continuous in $\mathring{F}$ endowed with the supremum norm $\|\cdot\|_\infty$. Observe that $0\in \mathring{F}$, using \cite[Proposition 3.1 and Proposition 4.1]{Ekeland} we obtain that $H^{**}(0)=H(0)$. Hence, we get the result.
\end{proof}
\begin{thm}\label{T-inf equals sup}
Let $(X,d)$ be a Polish metric space then problem $(B^*)$ has solutions.
\end{thm}

To prove this theorem, we will prove the following lemma
\begin{lem}\label{L-Lipschitz}
Let $(X,d)$ be a Polish metric space. For every $\lambda >0$, let $f\in F_{\lambda}$ then $S_\lambda f$ and $\left(S_{\lambda}\circ S_{\lambda}\right)f$ are $2\lambda  b^2D$-Lipschitz functions on $K$, where $D=\text{diam}(K)$.
\end{lem}
\begin{proof}
As $K$  is a compact subset of $X$ then $K$ is bounded and thus $D=\text{diam}(K)<\infty$. Let any $x_1,x_2\in K$. For every $\varepsilon>0$, there exists $y_0\in K$ such that $S_\lambda f \left(x_2\right)\geq \lambda b^2d^2\left(x_2,y_0\right)-f\left(y_0\right)-\varepsilon$. Moreover, it is clear that $S_\lambda f \left(x_1\right)\leq \lambda b^2d^2\left(x_1,y_0\right)-f\left(y_0\right)$. Hence, we get that \begin{align*}
S_\lambda f \left(x_1\right)-S_\lambda f \left(x_2\right)&\leq \lambda b^2\left[d^2\left(x_1,y_0\right)-d^2\left(x_2,y_0\right)\right]+\varepsilon\\
& \leq \lambda b^2 d\left(x_1,x_2\right)\left[d\left(x_1,y_0\right)+d\left(x_2,y_0\right)\right]+\varepsilon\\
&\leq 2\lambda b^2D d\left(x_1,x_2\right)+\varepsilon.
\end{align*}
Similarly, $S_\lambda f \left(x_2\right)-S_\lambda f \left(x_1\right)\leq 2\lambda b^2D d\left(x_1,x_2\right)+\varepsilon$. Therefore, $S_{\lambda}f$ is a $2\lambda b^2 D$-Lipschitz function. By the same arguments above, we also get that $\left(S_{\lambda}\circ S_{\lambda}\right)f$ is a $2\lambda b^2 D$-Lipschitz function.  
\end{proof}
\begin{proof}[Proof of Theorem \ref{T-inf equals sup}] 
Let $f^n=\left(f_1^n,\ldots,f_k^n\right)$ be a maximizing sequence for $(B^*)$. For each $i\in\{1,\ldots,k-1\}$ we define $\widetilde{f}^n_i:=\left(S_{\lambda_i}\circ S_{\lambda_i}\right)f^n_i$. Then $\widetilde{f}^n_i$ is bounded on $K$ for every $i=1,\ldots,k-1$. Since $S_{\lambda_i}f_i(x)\geq -\lambda_ia$ for every $x\in K$ and every $i=1,\ldots,k-1$, we get $\widetilde{f}^n_i(x)=\inf_{y\in K}\left\{\lambda_ib^2d^2(x,y)-S_{\lambda_i}f_i^n(y)\right\}\leq -S_{\lambda_i}f_i^n(x)\leq \lambda_ia$ for every $x\in K.$

Moreover, it is easy to see that $f_i^n\leq \widetilde{f}_i^n$ on $K$ and $S_{\lambda_i}\widetilde{f}^n_i=S_{\lambda_i}f^n_i$ for every $i=1,\ldots,k-1$. Hence $\overline{S}_{\lambda_i}\widetilde{f}^n_i=\overline{S}_{\lambda_i}f^n_i$ for every $i=1,\ldots,k-1$. For every $n\in \N$, we define $\widetilde{f}^n_k:=-\sum_{i=1}^{k-1}\widetilde{f}^n_i$. As $f_i^n\leq \widetilde{f}_i^n$ on $K$, one has $\widetilde{f}_k^n\leq -\sum_{i=1}^{k-1}f_i^n=f_p^n$. Thus, $\widetilde{f}_k^n(x)\leq \lambda_ka$ for every $x\in K$ and $S_{\lambda_k}\widetilde{f}_k^n\geq S_{\lambda_k}f^n_k$ for every $n\in \N$. Thus, $\overline{S}_{\lambda_k}\widetilde{f}_k^n\geq \overline{S}_{\lambda_k}f^n_k$ for every $n\in \N$. Therefore, we obtain that $$\limsup_{n\rightarrow \infty}\sum_{i=1}^k\int_K \overline{S}_{\lambda_i}\widetilde{f}_i^n(x)d\mu_i(x)\geq \lim_{n\rightarrow\infty}\sum_{i=1}^k\int_K\overline{S}_{\lambda_i}f_i^n(x)d\mu_i(x)=\sup(B^*).$$

Using Lemma \ref{L-Lipschitz} we get that $\widetilde{f}_i^n$ is a $2\lambda_i b^2D$-Lipschitz function on $K$ for every $i=1,\ldots,k-1$ and every $n\in\N$. As $\widetilde{f}^n_k:=-\sum_{i=1}^{k-1}\widetilde{f}^n_i$ and $\sum_{i=1}^k\lambda_i=1$, we obtain that $\widetilde{f}_k^n$ is a $2\left(1-\lambda_k\right)b^2D$-Lipschitz function on $K$. Then applying Ascoli-Arzela Theorem on compact set $K$ and using a standard diagonal argument there exists a subsequence of $\widetilde{f}^n=\left(\widetilde{f}_1^n,\ldots,\widetilde{f}_k^n\right)$ which we still denote by $\left\{\widetilde{f}^n\right\}$ such that $\widetilde{f}^n$ converges uniformly to $\widetilde{f}=\left(\widetilde{f}_1,\ldots,\widetilde{f}_k\right)$. Then $\widetilde{f}_i\in F_{\lambda_i}$ for every $i\in\{1,\ldots,k\}$. As $\sum_{i=1}^k\widetilde{f}_i^n=0$ for every $n\in \mathbb{N}$, we get that $\sum_{i=1}^k\widetilde{f}_i=0$. This yields,
\begin{align*}
\sum_{i=1}^k\int_K\overline{S}_{\lambda_i}\widetilde{f}_i(x)d\mu_i(x)&\leq \sup (B^*)\leq\sum_{i=1}^k \limsup_{n\rightarrow \infty}\int_K \overline{S}_{\lambda_i}\widetilde{f}_i^n(x)d\mu_i(x).
\end{align*}

Applying Fatou Lemma, we obtain that \begin{align*}
\sum_{i=1}^k\int_K\overline{S}_{\lambda_i}\widetilde{f}_i(x)d\mu_i(x) & \leq\sum_{i=1}^k \int_K\limsup_{n\rightarrow \infty} \overline{S}_{\lambda_i}\widetilde{f}_i^n(x)d\mu_i(x)\\
& =\sum_{i=1}^k \int_K\limsup_{n\rightarrow \infty}\left[\min\left\{ \inf_{y\in K}\left\{\lambda_ib^2d^2(x,y)-\widetilde{f}_i^n(y)\right\},\lambda_ia\right\}\right] d\mu_i(x)\\
&\leq \sum_{i=1}^k \int_K\min\left\{\inf_{y\in K}\left\{ \limsup_{n\rightarrow \infty}\left(\lambda_ib^2d^2(x,y)-\widetilde{f}_i^n(y)\right)\right\},\lambda_ia\right\} d\mu_i(x)\\
& = \sum_{i=1}^k \int_K\min\left\{\inf_{y\in K}\left\{\lambda_ib^2d^2(x,y)-\widetilde{f}_i(y)\right\},\lambda_ia\right\} d\mu_i(x)\\
& = \sum_{i=1}^k\int_K\overline{S}_{\lambda_i}\widetilde{f}_i(x)d\mu_i(x).
\end{align*}
Therefore, we must have equality everywhere. Hence, we get the result.
\end{proof}

\end{document}